\documentclass{amsart}
\usepackage{amssymb,latexsym}
\usepackage{amsmath}
\usepackage[dvipdfm]{graphicx}
 \usepackage{color}
\usepackage{enumerate}

\newcommand \url [1] {\texttt{#1}}
\newcommand \init [1] {#1}
\newcommand \etbf [1] {#1} 

\newenvironment{enumeratei}{\begin{enumerate}[\upshape (i)]} {\end{enumerate}}
\numberwithin{equation}{section}
\theoremstyle{plain}
 \newtheorem{theorem}{Theorem}[section]

 \newtheorem{lemma}[theorem]{Lemma}
 \newtheorem{proposition}[theorem]{Proposition}
 
 \newtheorem{corollary}[theorem]{Corollary}
\theoremstyle{definition}
 \newtheorem{definition}[theorem]{Definition}

 \newtheorem{remark}[theorem]{Remark}
\theoremstyle{remark}

\theoremstyle{plain} 
\theoremstyle{definition} 
\theoremstyle{remark} 

\newcommand \ssty [1] {\scriptstyle{#1}}
\newcommand \sssty [1] {\scriptscriptstyle{#1}}
\newcommand \tbf[1] {\textbf{#1}}

\renewcommand \emptyset{\varnothing}
\renewcommand\phi{\varphi}
\renewcommand\rho{\varrho}   
\renewcommand\epsilon{\varepsilon} 
\renewcommand \theta {\vartheta}   

\newcommand\semmi[1] {}
\newcommand\piros[1]{{\textcolor{red}{#1}}}

\newcommand\tudnivalo [1] {}  

\newcommand \temphint [1]{}

\newcommand \czgsch {}
\newcommand \czgonly {}
\newcommand \czgolub {}

\newcommand \balpha {{\boldsymbol{\alpha}}}
\newcommand \bbeta  {\boldsymbol{\beta}}
\newcommand \bgamma{\boldsymbol{\gamma}}
\newcommand \bdelta{\boldsymbol{\delta}}

\newcommand \jcongg {\boldsymbol{\theta}}

\newcommand \dual [1] {{{#1}{}^{\delta}}}
\newcommand \Jir [1] {\textup{Ji}\,#1} 
\newcommand \Jirn [1] {\textup{Ji}_0\,#1}  
\newcommand \Mir [1] {\textup{Mi}\,#1}
\newcommand \Mire[1] {\textup{Mi}_1\,#1}
\newcommand \length [1] {\textup{length}\,#1}
\newcommand \blength [1] {\textup{length}(#1)} 
\newcommand\ideal[1]{\mathord\downarrow #1}
\newcommand\filter[1]{\mathord\uparrow #1}
\newcommand \tope [1]    {1_{#1}} 
\newcommand \bote [1]    {0_{#1}} 

\newcommand \leftb [1]  {\textup{C}_{{\textup{l}}}(#1)} 
\newcommand \rightb [1] {\textup{C}_{{\textup{r}}}(#1)} 
\newcommand \bound [1] {\textup{Bnd}(#1)} 

\newcommand \Diag [1] {\textup{Dgr}(#1)}  
\newcommand \Nar [1] {\textup{Nar}(#1)}   

\newcommand \subnor [1]{{\textup{NsSub}}\,#1} 
\newcommand \Halocs {{\textup{CSL}}} 

\newcommand \smLatsign {\textup{SSL}}
\newcommand \multi [1] {{\ddot{#1}}}

\newcommand \smLat[1] {\smLatsign(#1)}
\newcommand \ismLat [1]{\smLatsign(#1)\fentjelkong1}
\newcommand \fentjelkong[1]{{}^{\kern- #1pt \mathord{\cong}}}
\newcommand \sampleclssign {\mathcal K}

\newcommand \dsampleclasps[1]  {\sampleclssign{}^{\kern0pt\delta}(#1)}
\newcommand \disampleclasps[1] {\sampleclssign{}^{\kern0pt\delta}(#1)\fentjelkong0}

\newcommand \LatCompSerSign {\textup{CSL}}
\newcommand \multiLatCompSerSign {\textup{CS}\multi{\textup{L}}}

\newcommand \dLcs  [1] {\LatCompSerSign{}^{\kern0pt\delta}(#1)}
\newcommand \dpLcs [1] {{}^{\kern0pt\delta}\kern-1pt\multi{\LatCompSerSign}(#1)}

\newcommand \diLcs [1] {\LatCompSerSign{}^{\kern0pt\delta}(#1)\fentjelkong1}
\newcommand \dipLcs[1] {{\multiLatCompSerSign}{}^{\kern0pt\delta}(#1)\fentjelkong1}
\newcommand \izoclass [1] {\mathbf I(#1)}

\newcommand \Seg [1] {\textup{Seg}(#1)}
\newcommand \blokk [2] {#1/#2}
\newcommand \bblokk [2] {(#1)/#2}
\newcommand \bbblokk [2] {\bigl(#1\bigr)/#2}
\newcommand \pblokk [2] {\textup{Seg}(#2, #1)}

\newcommand\quotientalg [2] {{#1}/{#2}  }
\newcommand\quotientset [2] {{#1}/{#2}  }

\newcommand \restrict [2] {{#1}\kern-1pt \rceil_{\kern-1pt #2}}
\newcommand\set [1]{\{#1\}}
\newcommand \bigset[1] {\bigl\{#1\bigr\}} 

\newcommand \osubs{{\mathord{\subseteq}}} 
\newcommand \osups{{\mathord{\supseteq}}} 

\newcommand \boper [1] {#1^{\kern-1pt\ssty{{\mathord=\kern-5pt\mathord{\parallel}}}}}

\newcommand \kernit {{\kern -6pt}}
\newcommand \tbu[1] {{ \phantom{\Big|} \kern -2pt \boxed{\tbf{#1}}  }}  
\newcommand \tbw[1] {{ \phantom{\Big|} \kern -2pt \boxed{\tbf{#1}}\kernit}}  
\newcommand \vakvec [1] {\vec{#1}}

\newcommand \secie {\boldsymbol{\rho}_{\kern -1pt e}^i}

\newcommand \bpbeta[1] {\boldsymbol{\beta}_{\kern-1pt #1}}
\newcommand \lat [1] {{#1}^{\sssty{\textup{lat}}} }
\newcommand \biglat [1] {{(#1)}{}^{\sssty{\textup{lat}}} }
\newcommand \bigDiag [1] {\textup{Diag}\bigl(#1\bigr)}
\newcommand \graf [1] { {#1}{}^{\sssty{\bullet} }}
\newcommand \isom  {\boldsymbol{\rho}_{\sssty{\mathord{\cong}}}}
\newcommand \joing {\mathrel{\mathord\vee_{\kern -2pt G}}}
\newcommand \joinl {\mathrel{\mathord\vee_{\kern -2pt L}}}
\newcommand \joinvg {\mathrel{\mathord\vee_{\kern -2pt G'}}}
\newcommand \diai [1] {#1^{{\kern-0.7pt\natural}}}
\newcommand \diah [1] {#1{}^{{\kern-0.7pt  \sssty{{\triangledown}} }}}
\newcommand \bigdiah [1] {\diah {\bigl(#1\bigr)} }
\newcommand \diag [1] {#1{}^{{\kern-0.7pt\ast}}}
\newcommand \diac [1] {(#1)^{{\kern-0.7pt\diamond}}}
\newcommand \diavarc [1] {#1^{{\kern-0.7pt\diamond}}}
\newcommand \scells [1]   {\textup{SCells}(#1)}
\newcommand \celldn [1] {{\textup{cell}_{\kern-0.3pt\sssty{\pmb{\mathord{\downarrow}}}}}(#1)}
\newcommand \cellup [1] {{\textup{cell}^{\kern-0.7pt\sssty{\mathord{\pmb{\uparrow}}}}}(#1)}

\newcommand \semipatch  {{ \,\pmb{\mathcal H}\kern 0.5pt}} 
\newcommand \patch [1] { \pmb{\mathcal P}_{\kern-2pt\textup{max}}(#1)} 
\newcommand \aslim [1] {#1^{\kern-1pt \bullet}} 

\newcommand \ilupphold [2] {\textup{sp}_{\sssty{\textup{left}}}^{\kern-2pt +\kern 2 pt #2}(#1)} 
\newcommand \irupphold [2] {\textup{sp}_{\sssty{\textup{right}}}^{\kern-2pt +\kern 2pt #2}(#1)} 
\newcommand \Ker[1] {\textup{Ker}\,#1} 
\newcommand \cproj {\mathrel{ \mathord{\Rightarrow} \kern-7.5pt \mathord{\Rightarrow} }} 
\newcommand \cpreq {\mathrel{ \mathord{\Leftarrow}  \kern-7.5pt \mathord{\Leftrightarrow} \kern-7.5pt \mathord{\Rightarrow} }} 
\newcommand \uppers {\,{\buildrel{\sssty{\textup{up}}}\over \rightarrow}\kern-8pt\mathord{\rightarrow}\; } 
\newcommand \upperpa [1] {\,{\buildrel{\sssty{\textup{up}}}\over \rightarrow}\kern-8pt\mathord{\rightarrow}_{#1}\; } 
\newcommand \dnpers {\,{\buildrel{\sssty{\textup{dn}}}\over \rightarrow}\kern-8pt\mathord{\rightarrow}\; } 
\newcommand \dnperpa [1] {\,{\buildrel{\sssty{\textup{dn}}}\over \rightarrow}\kern-8pt\mathord{\rightarrow}_{#1}\; } 
\newcommand \drestrict [2] {{#1}\rceil_{\kern-1pt #2}} 

\begin{document}
\title[Composition series and slim semimodular lattices]
{Composition series in groups and the structure of slim semimodular lattices}
\author[G.\ Cz\'edli]{G\'abor Cz\'edli}
\email{czedli@math.u-szeged.hu}
\urladdr{http://www.math.u-szeged.hu/$\sim$czedli/}
\address{University of Szeged\\Bolyai Institute\\
Szeged, Aradi v\'ertan\'uk tere 1\\HUNGARY 6720}

\author[E.\,T.\,Schmidt]{E.\,Tam\'as Schmidt}
\email{schmidt@math.bme.hu} 
\urladdr{http://www.math.bme.hu/$\sim$schmidt/}
\address{Mathematical Institute of the Budapest University of
    Technology and Economics\\
          M\H{u}egyetem rkp.~3\\
          H-1521 Budapest\\
          Hungary}

\thanks{This research was supported by the NFSR of Hungary (OTKA), grant numbers  K77432 and
K83219, and by  T\'AMOP-4.2.1/B-09/1/KONV-2010-0005}



\date{May 4, 2011; revised January 9, 2013}

\begin{abstract} Let $\vakvec H$ and $\vakvec K$ be finite composition series of a group $G$. 
The intersections $H_i\cap K_j$ of their members form a lattice $\Halocs(\vakvec H,\vakvec K)$ under set inclusion. 
Improving the Jordan-H\"older theorem, G.\ Gr\"atzer, J.\,B.\ Nation and the present authors  have recently shown that  $\vakvec H$ and $\vakvec K$ determine a unique permutation $\pi$
such that,   for all  $i$, the $i$-th factor of $\vakvec H$ is ``down-and-up projective'' to the $\pi(i)$-th factor of $\vakvec K$.
We prove that $\pi$ determines the lattice $\Halocs(\vakvec H,\vakvec K)$. 
More generally, we describe slim semimodular lattices, up to isomorphism, by~permutations, up to an equivalence relation called  ``sectionally inverted or equal''. As a consequence, we prove that the abstract class of all $\Halocs(\vakvec H,\vakvec K)$ coincides with the class of duals of all slim semimodular lattices.
\end{abstract}
\maketitle
\section{Introduction}\label{section:intRo}
\subsection{{Composition series and lattices}}
Let $\vakvec H\colon \set1= H_0 \triangleleft H_1 \triangleleft \cdots \triangleleft H_n =G\label{f:nsera}$ and $\vakvec K\colon \set1=K_0 \triangleleft K_1  \triangleleft \cdots \triangleleft K_n =G$ be  composition series of a group $G$. Denote $\bigset{H_i\cap K_j: i,j\in\set{0,\ldots,n}}$ by $\Halocs(\vakvec H,\vakvec K)$, or by $\Halocs_n(\vakvec H,\vakvec K)$ if we want to specify the common length $n$ of the composition series. Clearly, 
\[\Halocs(\vakvec H,\vakvec K)=\bigl(\Halocs(\vakvec H,\vakvec K);\subseteq\bigr)\]
is a lattice, not just an order. (Orders are also called posets, that is, \underbar partially \underbar{o}rdered \underbar{sets}. The acronym $\Halocs$ comes from ``Composition Series Lattice''.) 
As usual, the relation ``{subnormal subgroup}'' is the transitive closure of the relation ``normal subgroup''. 
For subnormal subgroups $A \triangleleft B$ and $C\triangleleft D$ of  $G$, the quotient $B/A$ will be called \emph{subnormally down-and-up projective} to $D/C$, if there are subnormal subgroups $X\triangleleft Y$ of $G$ such that 
\begin{equation}\label{eqDnoUPPpsrpsccccb}
AY=B,\quad A\cap Y=X,\quad CY=D,  \quad C\cap Y=X \text.
\end{equation}
Clearly, $B/A\cong D/C$ in this case, because both groups are isomorphic with $Y/X$. 
Since $G$ is of finite composition length, its subnormal subgroups form a sublattice $\subnor G=(\subnor G;\subseteq)$ of the lattice of all subgroups by a classical result of H.\ Wielandt~\cite{r:wieland}; see also  R.\ Schmidt~\cite[Theorem 1.1.5]{r:rschmidt} and the remark after its proof,  or  M.~Stern~\cite[p.~302]{r:stern}.

It is not hard to see that 
 $\subnor G$  is \emph{dually semimodular} (also called \emph{lower semimodular}); see  \cite[Theorem 2.1.8]{r:rschmidt}, or the  proof of  \cite[Theorem 8.3.3]{r:stern}, or the proof of  J.\,B.\ Nation~\cite[Theorem 9.8]{r:nationbook}. Since this property depends only on the meet operation and $\Halocs(\vakvec H,\vakvec K)$ is a meet-subsemilattice of  $\subnor G$, we conclude that $\Halocs(\vakvec H,\vakvec K)$ is a dually semimodular  lattice. Note, however, that $\Halocs(\vakvec H,\vakvec K)$ is \emph{not} a sublattice of $\subnor G$ in general; this is witnessed by the eight-element elementary 2-group $(\mathbb Z_2;+)^3$.

A lattice is \emph{dually slim} if it is finite and it has no three pairwise incomparable meet-irreducible elements. Since  each meet-irreducible element of $\Halocs(\vakvec H,\vakvec K)$ occurs in $\vakvec H$ or $\vakvec K$, it follows that  $\Halocs(\vakvec H,\vakvec K)$ is a \emph{dually slim} lattice.

We  proved the following result in \czgsch\cite{rczgschjordhold}.
\begin{theorem}\label{jhbBNbeMbzTk}
There exists a \emph{unique} permutation $\pi$ of the set $\{1,\ldots,n\}$ such that $H_i/H_{i-1}$ is subnormally down-and-up projective to $K_{\pi(i)}/K_{\pi(i)-1}$, for $i=1,\ldots,n$.
\end{theorem}
This permutation will be described later in Remark~\ref{whIsPefRjHOdr}.
Note that, as opposed to the uniqueness of $\pi$, the subnormal subgroups $X$ and $Y$ occurring in \eqref{eqDnoUPPpsrpsccccb} are not unique, in general, and they need not belong to $\Halocs(\vakvec H,\vakvec K)$.
Note also that even the statement on the 
\emph{existence} of $\pi$, due to G.\ Gr\"atzer and J.\,B.\ Nation~\cite{refgrnation}, strengthens  the classical Jordan-H\"older Theorem, see  
 C.\ Jordan~\cite{r:jordan} and O.\ H\"older~\cite{r:holder}.

One of our goals is to show that $\pi$ determines the lattice $\Halocs(\vakvec H,\vakvec K)$, see Corollary~\ref{corcOmpseDeiTmsnd}. We will also show that the 
lattices of the form 
$\Halocs(\vakvec H,\vakvec K)$ are characterized as  duals of slim semimodular lattices, see Corollary~\ref{RcThTwocLsSuQal}. 
These results follow from our main result, Theorem~\ref{thmmAin}, which is purely lattice theoretic.

\subsection{{Slim semimodular lattices and matrices}}
A \emph{slim lattice} is a finite lattice  $M$ such that $\Jirn M$, the order of its join-irreducible elements (including 0), contains no three-elem\-ent antichain. This concept is due to G.\ Gr\"atzer and E.\ Knapp~\cite{r:GratzerKnapp1}. 
By R.\,P.\ Dilworth~\cite{r:Dilw-paper}, a finite lattice $M$ is slim if{f} $\Jirn M$  is the  union of two chains.  
\semmi{A lattice
$M$ is called (upper) \emph{semimodular}, if 
$a \prec  b$ implies that $a\vee c\preceq  b\vee c$,  for all $a,b,c\in M$; 
see  M.\ Stern~\cite{r:stern}, and see also G.\ Gr\"atzer~\cite{rGratzerGLT} and \cite{r:Gr-LTFound}, and \czgsch\cite{r:czgschsomeres} for more about these lattices.}

By   \czgsch{}\cite[Lemma~6]{rczgschjordhold}, slim lattices are \emph{planar}. So they are easy objects to understand. Slim semimodular lattices come up in 
proving Theorem~\ref{jhbBNbeMbzTk} and also in the finite congruence lattice representation problem; see, for example, G.\ Cz\'edli~\czgonly\cite{r:czg-clrp}, G.\ Gr\"atzer and E.\ Knapp~\cite{r:GratzerKnapp3} and \cite{r:GratzerKnapp4}, and E.\,T.\ Schmidt~\cite{r:sch-injrep}. 
Several ways of describing slim semimodular lattices were developed. Two visual (recursive) methods of constructing slim semimodular lattices were given in \czgsch{}\cite{r:czg-sch-visual}. Furthermore, these lattices were characterized by matrices in \czgonly\cite{r:czg-mtx}.
Let $\ismLat h$ denote the set of isomorphism classes of slim semimodular lattices of  length $h$.
Based on the matrix characterization given in \czgonly\cite{r:czg-mtx}, G.\ Cz\'edli, L.\ Ozsv\'art and B.\ Udvari~\czgolub\cite{r:czg-o-u}
succeeded in calculating the number~$|\ismLat h|$  of (isomorphism classes) of slim semimodular lattices of a given length $h$; the  value of $|\ismLat h|$ has been computed  up to $h=100$. 

The matrices in \cite{r:czg-mtx} correspond to bijective partial maps. Although they yield an optimal description in some sense, their definition is a bit complicated. 
Our goal is to describe slim semimodular lattices by (totally defined) bijective maps; namely, by permutations. The fact that three different ideas lead to the same permutations indicate that these permutations are natural objects. 
As opposed to the matrices, our 
permutations say something interesting of the  magnitude of the number $|\ismLat h|$; indeed, our main theorem trivially yields that $h!$ is an upper bound for $|\ismLat h|$.  Furthermore, the present approach yields  Corollaries~\ref{corcOmpseDeiTmsnd} and \ref{RcThTwocLsSuQal}, while the matrix approach  does not.

\subsection{{Planar diagrams}}
To avoid ambiguity, we have to distinguish between planar lattices and their diagrams. 
Let $\diag M$  be a planar diagram of a finite (planar) lattice $M$. 
For $u\leq v\in M$, let  $\diag{[u,v]}$ denote the unique diagram of the interval $[u,v]$ 
{determined by $\diag M$.}
The edges of $\diag M$ divide the plane into regions; the minimal regions are called \emph{cells}. By a \emph{covering square} we mean a four-element cover-preserving sublattice of length two. Cells that are covering squares are called \emph{$4$-cells}. 
The \emph{left boundary chain}, the \emph{right boundary chain} and the \emph{boundary} of $\diag M$ are denoted by $\leftb{\diag M}$, $\rightb{\diag M}$ and  $\bound {\diag M}=\leftb{\diag M}\cup\rightb{\diag M}$, respectively.

Next, let $\diah M$  also be a planar diagram of  $M$. Then $\diag M$ and $\diah M$ are called \emph{boundarily similar diagrams} of $M$, if $\leftb{\diag M} = \leftb{\diah M}$ and  
$\rightb{\diag M} = \rightb{\diah M}$.
(Notice that if two diagrams are similar in the sense of  D.\ Kelly and I.\ Rival~\cite{r:kelly-rival}, then they are boundarily similar, but not conversely.)
More generally, if $\diag {M_i}$ is a planar diagram of $M_i$, then 
$\diag {M_1}$ is boundarily similar to $\diag {M_2}$ if there is a lattice isomorphism $\gamma\colon M_1\to M_2$ such that $\gamma(\leftb{\diag {M_1}}) = \leftb{\diag {M_2}}$ and $\gamma(\rightb{\diag {M_1}}) = \rightb{\diag {M_2}}$. 
We will consider diagrams only up to boundary similarity. 

Let $\Diag M$ denote the \emph{set of all planar diagrams} of $M$. Then $\Diag M$ is a finite set since boundarily similar diagrams are considered equal. 
{Sometimes} we need a notation, $\lat D$, which is the lattice $\lat D$ from its  diagram $D$. Note that $M=\biglat{\diag M}$ for every planar lattice $M$ and any $\diag M\in\Diag M$.

Let $L$ be a slim semimodular lattice  of length $n$. Although it is  $L$ we want to characterize  by permutations, in this section, we work with a fixed diagram $\diag L$ of $L$. 
The elements of $\bound{\diag L}$ will be denoted as follows: 
\begin{align}\label{eleriboudelesnot}
\begin{aligned}
\leftb {\diag L}&=\set{0=c_0\prec c_1\prec\cdots\prec c_n=1},\cr \rightb{\diag L}&=\set{0=d_0\prec d_1\prec\cdots\prec d_n=1}\text.
\end{aligned}
\end{align}

An element of $L$ is called a \emph{narrows}  if it  is comparable with all elements of $L$. This terminology is  from G.\ Gr\"atzer and R.\,W.\ Quackenbush~\cite{r:gr-qbush}; however, as opposed to \cite{r:gr-qbush}, we define 0 and 1 as narrows of $L$. 
The set of narrows is denoted by $\Nar L$. The elements of $\Nar L\setminus\set{0,1}$ are called \emph{nontrivial narrows} of $L$. 
For  $\diag L\in \Diag L$, we define  $\Nar{\diag L}:=\leftb {\diag L}\cap\rightb{\diag L}$; clearly, $\Nar{\diag L}=\Nar L$.
Note that $\Nar L$  is a chain
. 
A finite lattice $M$ is called (\emph{glued sum}) \emph{indecomposable} if $|M|=1$ or $2=|\Nar M|<|M|$. 

The set of all meet-irreducible elements (including $1$) is denoted by $\Mire M$. Let $\Jir M=\Jirn M\setminus\set 0$ and $\Mir M=\Mire M\setminus\set 1$. 
\czgsch~\cite[Lemma 7]{r:czg-sch-visual} asserts that $\bound {\diag L}$  is the same for all $\diag L\in\Diag L$. Hence, we can define $\bound L$ as  $\bound {\diag L}$, where $\diag L\in\Diag L$. 
By G.~Gr\"atzer and E.~Knapp~\cite[Lemma 4]{r:GratzerKnapp1}, 
\begin{equation}\label{eeMntcotMsttw}
\text{every element of }L\text{ is covered by at most two elements.} 
\end{equation}
By \cite[Lemma 8]{r:GratzerKnapp1}, $L$ is a so-called \emph{$4$-cell lattice}; this means that  all cells of every  $\diag L\in \Diag L$  are  $4$-cells. Furthermore, by \czgsch~\cite[Lemma 6]{r:czg-sch-visual} and \cite[Lemma 7]{rczgschjordhold}
\begin{gather}
\Jirn L\subseteq \bound L,
\label{ejirsubbound}\\
\text{the }4\text{-cells of }\diag L\text{ and the covering squares of }L\text{ are the same.}\label{eslfcelltsmecosQre}
\end{gather}
As usual, the set of permutations acting on $\set{1,\ldots,n}$ is denoted by $S_n$. The ordering  $1<\cdots<n$ of the underlying set 
will be important.

\begin{figure}
\centerline
{\includegraphics[scale=1.0]{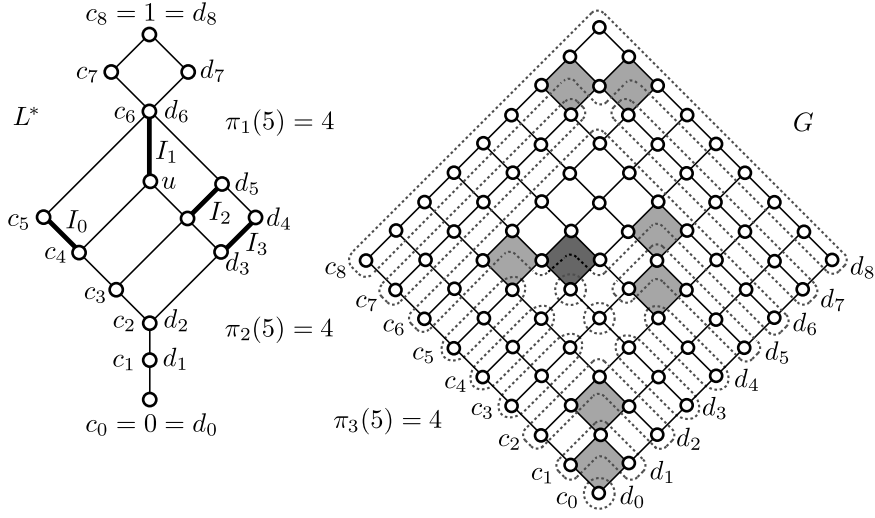}}
\caption{Three ways to define a permutation\label{figureone}}
\end{figure}

\section{Three ways to associate a permutation with a planar diagram}\label{sectionthreeways}

\begin{definition}\label{defpermone}
Let $L$ be a slim semimodular lattice. For a diagram $\diag L\in\Diag L$, we use the  notation introduced in \eqref{eleriboudelesnot}. 
We associate a permutation $\pi_1\in S_n$ with $\diag L$ as follows; see Figure~\ref{figureone} for an illustration. Let $i\in \set{1,\ldots,n}$. Take the prime interval $I_0:=[c_{i-1},c_i]$ on the left boundary. If $I_t$ is defined and it is on the left boundary of a $4$-cell, then let $I_{t+1}$ be the opposite edge of this $4$-cell. Otherwise, $I_{t+1}$ is undefined. The sequence $I_0,I_1, \ldots I_m$ of all the defined $I_t$-s is called a \emph{trajectory}. It goes from left to right, and it stops at the right boundary. Let  $I_m =[d_{j-1},d_j]$. We define $\pi_1(i):=j$. For $\diag L$ on the left of  Figure~\ref{figureone}, $i=5$, and $m=3$, the trajectory in question consists of the thick edges. 
\end{definition}

We consider $\diag L$ up to boundary similarity, and only  $\leftb{\diag L}$ and $\rightb{\diag L}$ are fixed. Hence it is not so clear  how the  trajectory goes in the  ``unknown interior'' of $\diag L$.
However, based on \eqref{eslfcelltsmecosQre}, it was proved in \czgsch~\cite{rczgschjordhold}  that $\pi_1$ is a uniquely defined map and it is a permutation. In fact, \cite{rczgschjordhold} proves  an appropriate uniqueness result for  any two maximal chains without assuming slimness.

The definition of $\pi_1$ is quite visual. The next one is less visual but conceptually simpler. 
{As usual, $\ideal u$ stands for  $\set{x\in L: x\leq u}$, and $\filter u$ is defined dually.}

\begin{definition}\label{defpermmirr} 
We associate a permutation $\pi_2\in S_n$ with ${\diag L}$ as follows; see Figure~\ref{figureone}  again for an illustration. Let $i\in \set{1,\ldots,n}$. Take a meet-irreducible element $u\in L$ such that $c_i$ is the smallest element of $\leftb {\diag L}\setminus \ideal u$. Let $d_j$ be the smallest element of $\rightb{\diag L}\setminus \ideal u$. We define $\pi_2(i):=j$. 
\end{definition}

\begin{lemma}\label{lemdeftwoworks} 
$\pi_2$ is uniquely defined and belongs to $S_n$. Furthermore, the element~$u$ in  Definition~\ref{defpermmirr} is uniquely determined.
\end{lemma} 

\begin{proof} Let $B_i=\filter{c_{i-1}}\setminus \filter{c_i}$. It is not empty since it contains $c_{i-1}$. By \eqref{ejirsubbound}, each element of $B_i$ is of the form $c_s\vee d_t$, and we can clearly assume that $s=i-1$. Since $\rightb{\diag L}$ is a chain, we conclude that $B_i$ is also a chain. Let $u$ be the largest element of $B_i$. Obviously, $u\in\Mir L$, whence $u$ satisfies the requirements of Definition~\ref{defpermmirr}.  Assume that so does $v$. Then  $v\in B_i\cap\Mir L$ and $v\leq u$. 
By semimodularity, $v=c_{i-1}\vee v\preceq c_i\vee v$. Clearly, $ B_i\not\ni c_i\vee v  \not\leq u$, implying that $v=u\wedge (c_i\vee v )$. Hence $v=u$ since $v$ is meet-irreducible. This proves the uniqueness of $u$ in the definition. Therefore, $\pi_2$ is a uniquely defined $\set{1,\ldots,n}\to\set{1,\ldots,n}$ map. Since $\pi_2$ depends only on the assignment of the left and right boundary chains and on the meet operation,  boundarily similar diagrams of $L$ yield the same $\pi_2$.

Interchanging left and right in the definition, we obtain a uniquely defined map $\sigma\colon \set{1,\ldots,n}\to\set{1,\ldots,n}$ analogously. 
That is, $\sigma(j)=i$ if{f} there is a $u\in \Mir L$ such that $d_j$ and $c_i$ are the  smallest elements of  $\rightb{\diag L}\setminus\ideal u$ and $\leftb{\diag L}\setminus\ideal u$, respectively. 
The uniqueness of $u$ (both in the definition of $\pi_2$ and that of $\sigma$) clearly yields that the composite maps $\pi_2\circ \sigma$ and $\sigma\circ \pi_2$ are the identity maps. Thus, $\pi_2$ is a permutation.
\end{proof}

The following corollary is evident by the second sentence of Lemma~\ref{lemdeftwoworks}. It also  follows easily from known results on convex geometry, see  
R.\,P.\ Dilworth~\cite{r:dilworth40} or K.\ Adaricheva, V.\,A.\ Gorbunov and V.\,I.\ Tumanov~\cite[Theorem 1.7.(1-2)]{r:adarichevaetal}.
\semmi{Yet, to point out a link with  convex geometries, we present an independent proof.}

\begin{corollary}\label{corlmiszma} For every slim semimodular lattice $K$, $|\Mir K|=\length K$.
\end{corollary}

\semmi{
\begin{proof} \piros{By semimodularity and  \eqref{eeMntcotMsttw}, $K$ is (locally) upper distributive; that is, for each $a\in K$, the interval spanned by $a$ and the join of its covers is a distributive lattice. 
Hence, by the dual of 
K.\ Adaricheva, V.\,A.\ Gorbunov and V.\,I.\ Tumanov~\cite[Theorem 1.7.(1-2)]{r:adarichevaetal}, proved originally by R.\,P.\ Dilworth~\cite{r:dilworth40}, each element of $K$ has a unique irredundant decomposition into a meet of meet-irreducible elements. 
Let $E=\set{1=e_0\succ e_1\succ\cdots\succ e_n=0}$ be a maximal chain in $K$. Let $E_i= \Mir K \cap \filter{e_i}$, for  $i\in\set{0,\ldots,n}$. 
Since $E_0=\emptyset$ and $E_n=\Mir K$, it suffices to show that $|E_i\setminus E_{i-1}|=1$ for $i\in\set{1,\ldots,n}$. Suppose to the contrary that an $i\in\set{1,\ldots,n}$ violates this equation. Clearly, there are distinct elements $x_1,x_2\in E_i\setminus E_{i-1}$. Obviously, for each $j\in\set{1,2}$,
$E_{i-1}\cup \set{x_j}$ has a minimal subset $A_j$ whose meet is $e_i$. Since $e_i=\bigwedge A_j$ is an irredundant decomposition, the uniqueness mentioned above yields that $A_1=A_2$. This contradicts $x_1\in A_1\setminus A_2$.}
\end{proof}
}

The third way of defining a permutation  is more complicated than the other two. However, it will play the main role in the proof of Theorem~\ref{thmmAin}. The prerequisites below are taken from \czgonly\cite{r:czg-mtx} and \czgsch{}\cite{rczgschhowto}. 

By a \emph{grid} we mean the direct product of two finite chains. If these chains are of the same size, then we speak of a \emph{square grid}. If $G$ is a square grid, then the elements of its lower left boundary and those of the lower right boundary are denoted by 
\begin{equation}\label{egRidnostdnIon}
C=\set{0=c_0\prec c_1\prec\cdots\prec c_n},\qquad 
D=\set{0=d_0\prec d_1\prec\cdots\prec d_n}\text, 
\end{equation}
respectively, and we say that $G$ is the square grid of length $2n$; see Figure~\ref{figureone} for $n=8$.
Note that each element of the grid can be written uniquely in the form $c_i\vee d_j$ where $i,j\in\set{0,\ldots,n}$. For lattices $M_1$ and $M_2$, a join-(semilattice)-homomorphism 
$\phi\colon M_1\to M_2$ is called \emph{cover-preserving} if $x\prec y$ implies that $\phi(x)\preceq\phi(y)$, for all $x,y\in M_1$. Kernels of this sort of homomorphisms are called \emph{cover-preserving join-congruences}. 

Let $M$ be a slim semimodular lattice, and let $u\in M$. If there is a unique $4$-cell whose top, resp.\ bottom, is $u$, then it is denoted by $\celldn u$, resp.\ $\cellup u$. Take a $4$-cell $B=\set{\bote B=a\wedge b,a,b,\tope B=a\vee b}$ of $M$. 
Then $B=\cellup{\bote B}=\cellup{a\wedge b}$ by \eqref{eeMntcotMsttw}, but the notation $\celldn{\tope B}$ is not always allowed. 
Consider a  join-congruence $\balpha $ of $M$. 
We say that $B$ is an \emph{$\balpha $-forbidden $4$-cell} if  the $\balpha $-classes 
$\blokk a\balpha $, $\blokk b\balpha $ and $\bblokk {a\wedge b}\balpha $ are
pairwise distinct but either $\bblokk {a\vee b}\balpha  = \blokk a\balpha $ or $\bblokk {a\vee b}\balpha  = \blokk b\balpha $. Recall from \czgsch\cite{rczgschhowto} that, for any join-congruence $\balpha $ of $M$, 
\begin{equation}\label{ecoprecoiffnoforbs} 
\balpha \text{ is cover-preserving if{}f }M\text{ does not have an }\balpha \text{-forbidden }4\text{-cell.}
\end{equation}
If $\set{a,b}\subseteq \bblokk{a\vee b}\balpha \not\ni a\wedge b$, then $B$ is called a \emph{source cell} of $\balpha $. The set of source cells of $\balpha $ is denoted by $\scells \balpha $. The source cells are usually shaded grey. We are now ready to formulate

\begin{definition}\label{defpermmtxos}
We associate a permutation $\pi_3\in S_n$ with ${\diag L}$ as follows; see Figure~\ref{figureone} for an illustration. Let $G=\leftb{\diag L}\times\rightb{\diag L}$. Let us agree that $\leftb{\diag L}$ and $\rightb{\diag L}$ are (identified with) the lower left boundary and the lower right boundary of $G$, respectively. Using the notation \eqref{eleriboudelesnot}, the kernel of the join-homomorphism $\eta\colon G\to L$, defined by $c_i\joing d_j\to c_i\joinl d_j$, will be denoted by $\bbeta_{\diag L}$. 
For $i\in \set{1,\ldots,n}$, we define $j=\pi_3(i)$ by the property that $\celldn{c_i\joing   d_j}\in \scells {\bbeta_{\diag L}}$. 
\end{definition}

\begin{lemma}\label{defperMbymtx}
$\pi_3$ is uniquely defined and belongs to $S_n$. 
\end{lemma}

\begin{proof} 
Note that the quotient join-semilattice $\quotientalg  G{\bbeta_{\diag L}}$ is actually a lattice since it is a finite join-semilattice with 0. 
Note also that $\quotientalg G{\bbeta_{\diag L}}\cong L$ by the Homomorphism Theorem, see S.\ Burris and H.\,P.\ Sankappanavar~\cite[Thm.\ 6.12]{r-burrissankap}. 
Since $\eta$ acts identically on $\leftb{\diag L}$, the  $\bbeta_{\diag L}$-classes $\blokk {c_i}{\bbeta_{\diag L}}$, $i=0,\ldots, n$, are pairwise distinct. 
We know from  \czgsch{}\cite[proof of Cor.\ 2]{rczgschhowto} that  $\bbeta_{\diag L}$ is cover-preserving. (Note that we know that  $L$ is a cover-preserving join-homomorphic image of a grid also from G.\ Gr\"atzer and E.\ Knapp~\cite{r:GratzerKnappAU} and M.\ Stern~\cite{r:stern}.)
Hence we conclude that 
\begin{align}
\begin{aligned}
\label{ecidjseMoDbEtakd}
\blokk {c_0}\bbeta_{\diag L} & \prec \blokk {c_1}\bbeta_{\diag L}\prec \cdots\prec \blokk {c_n}\bbeta_{\diag L}\,\text{,}\cr
\blokk {d_0}\bbeta_{\diag L} & \prec \blokk {d_1}\bbeta_{\diag L}\prec \cdots\prec \blokk {d_n}\bbeta_{\diag L}\text. 
\end{aligned}
\end{align}
Taking into account that  $\blength{\quotientalg G{\bbeta_{\diag L}}}=\length L=n$, we obtain that 
\begin{equation}\label{eneNWbchocclpsd}
\bblokk{c_i\vee d_n}{\bbeta_{\diag L}}=\bblokk{c_n\vee d_j}{\bbeta_{\diag L}}=\blokk 1{\bbeta_{\diag L}}\text{, for all }i,j\in\set{0,\ldots,n}\text.
\end{equation}
Consider the sequence 
\[[c_{i-1}\vee d_0, c_i\vee d_0],\,\, [c_{i-1}\vee d_1, c_i\vee d_1],\,\, \ldots,\,\, [c_{i-1}\vee d_n, c_i\vee d_n]
\] of prime intervals of $G$. By \eqref{eneNWbchocclpsd} and \eqref{ecidjseMoDbEtakd},  the last member of this sequence is collapsed while the first one is not collapsed by $\bbeta_{\diag L}$. Hence there is a $j\in\set{1,\ldots, n}$ such that $(c_{i-1}\vee d_{j-1}, c_{i}\vee d_{j-1}) \notin\bbeta_{\diag L}$ but $(c_{i-1}\vee d_{j}, c_{i}\vee d_{j}) \in\bbeta_{\diag L}$. In fact, there is exactly one $j$ since, for $t=j+1,\ldots, n$,  
\begin{equation}\label{iFjOKtNnfOAlToS}
(c_{i-1}\vee d_{j}, c_{i}\vee d_{j}) \in\bbeta_{\diag L}\text{ implies that } 
(c_{i-1}\vee d_{t}, c_{i}\vee d_{t})\in\bbeta_{\diag L}
\text.
\end{equation}
By \eqref{ecoprecoiffnoforbs}, $G$ has  no $\bbeta_{\diag L}$-forbidden square. Hence we conclude that $\celldn{c_i\vee d_j}\in\scells {\bbeta_{\diag L}}$, and this $j$ is unique by \eqref{iFjOKtNnfOAlToS}. By the left-right symmetry, for each $j\in\set{1,\ldots,n}$ there is exactly one $i\in\set{1,\ldots,n}$ such that  $\celldn{c_i\vee d_j}\in\scells {\bbeta_{\diag L}}$. Hence $\pi_3$ is a uniquely defined permutation on $\set{1,\ldots,n}$.
\end{proof}

\begin{proposition}\label{prdiagLpermalleqal} Let $\pi_1$, $\pi_2$, and $\pi_3$ denote the permutations associated with $\diag L$
in Definitions~\ref{defpermone}, \ref{defpermmirr}, and \ref{defpermmtxos}, respectively. Then $\pi_1=\pi_2=\pi_3$.
\end{proposition}

For $\pi_1=\pi_2=\pi_3$, we use the notation 
$\pi=\pi_{\diag L}$.

\begin{proof}[Proof of Proposition~\ref{prdiagLpermalleqal}] Assume that $j=\pi_3(i)$, that is, 
 $\celldn{c_i\joing  d_j}$ belongs to $\scells {\bbeta_{\diag L}}$. 
Let $u:=c_{i-1}\joinl d_{j-1}$ and $v:=c_{i}\joinl d_{j}$ (in $L$).
By the definition of $\eta$ and $\bbeta_{\diag L}$,   this means that
\begin{equation}\label{esiIWcjqpQlS}
u\neq v=c_{i-1}\joinl d_j = c_i\joinl d_{j-1}\text.
\end{equation}
Assume that $x\in L$ such that $u<x$. We know from \eqref{ejirsubbound} that $x$ is of the form $c_s\joinl d_t$. Since $x=u\joinl x$, we can assume that $i-1\leq s$ and $j-1\leq t$.
Hence \eqref{esiIWcjqpQlS} yields that $v\leq x$. 
This means that $v$ is the only cover of $u$, whence  $u\in \Mir L$. If  $c_i\leq u$, then 
\[u=c_i\joinl u=c_i\joinl c_{i-1}\joinl d_{j-1}=  c_i\joinl d_{j-1}\]  contradicts \eqref{esiIWcjqpQlS}. Therefore, $c_i$ is the smallest element of $\leftb{\diag L}\setminus \ideal u$. Similarly, $d_j$ is the smallest element of  $\rightb{\diag L}\setminus \ideal u$. Hence $j=\pi_2(i)$. Thus, $\pi_2$ equals $\pi_3$.

Next, assume that $j=\pi_1(i)$. Consider the trajectory $I_0,\ldots,I_m$ as in Definition~\ref{defpermone}. For $t=0,\ldots,m$, let 
$x_t$ and $y_t$ denote the bottom and the top of $I_t$, respectively. That is, $I_t=[x_t,y_t]$.  By \czgsch{}\cite[Lemmas 11 and 12]{rczgschjordhold}, there is a $k\in\set{0,\ldots,m}$ such that 
\begin{equation}\label{ewHretTrajTurndwn}
y_k=c_i\vee x_k,\quad c_{i-1}= c_i\wedge  x_k,\quad  
y_k=d_j\vee x_k,\quad d_{j-1}= d_j\wedge  x_k\text.
\end{equation}

We claim that $x_k$ is meet-irreducible.
If $m=0$, then $[c_{i-1},c_i]=[d_{i-1},d_i]=[x_k,y_k]\subseteq \Nar L$, whence $x_k\in \Mir L$. Hence we can assume that $m\geq 1$. Observe that  $y_k$ is join-reducible by \eqref{ewHretTrajTurndwn}. 
If $k\in\set{0,m}$, then 
$I_k$ is on the boundary of $L$, and  
the join-reducibility of $y_k$ together with
\czgsch{}\cite[Lemma 4]{r:czg-sch-visual} yields that $x_k\in\Mir L$. Hence we can assume that $k\in\set{1,\ldots,m-1}$. Then, by \eqref{eslfcelltsmecosQre}, we have two adjacent  $4$-cells: $B'=\set{x_{k-1}, y_{k-1}, x_k, y_k}$ and $B''=\set{x_{k}, y_{k}, x_{k+1}, y_{k+1}  }$. Suppose that $x_k$ is meet-reducible. Then it has a cover $v$ that is distinct from $y_k$. Clearly, $v\notin \ideal{y_k}$.   In the diagram $\diag L$, let $X_a=\leftb {\ideal{y_k}}$ and  $X_b=\rightb {\ideal{y_k}}$. Fix a maximal chain $X_0$ in $\filter{y_k}$. Then $Y_a:=X_a\cup X_0$ and $Y_b:=X_b\cup X_0$ are maximal chains of $\diag L$. For each maximal chain $Y$ and each $y\in \diag L$, exactly one of the following three possibilities holds: $y$ is strictly on the left of $Y$, or $y$ is strictly on the right of $Y$, or $y\in Y$.

Since  $v\parallel y_k$ and thus $v\notin Y_a\cup Y_b$, $v$ is either strictly on the left or strictly on the right of $Y_a$, and the same holds for $Y_b$.
If $v$ is both strictly on the right of $Y_a$ and  strictly on the left of $Y_b$, then $v\in \ideal {y_k}$ is a contradiction. Hence, by the left-right symmetry, we can assume that $v$ is strictly on the right of $Y_b$. 
However, $x_k$ is \emph{strictly} on the left of $Y_b$ since  $B'$ and $B''$ are adjacent $4$-cells. Therefore, see 
D.\ Kelly and 
I.\ Rival~\cite[Lemma 1.2]{r:kelly-rival},  there is a $w\in Y_b$ such that $x_k < w < v$. This contradicts that $x_k\prec v$, proving that $x_k$ is meet-irreducible.

Finally, \eqref{ewHretTrajTurndwn} implies that $c_i$ and $d_j$ are the smallest elements of $\leftb{\diag L}\setminus\ideal{x_k}$ and $\rightb{\diag L}\setminus\ideal{x_k}$, respectively. Hence $j=\pi_2(i)$, proving that $\pi_1$ equals $\pi_2$.
\end{proof}

The dual $\bigl(\Halocs(\vakvec H,\vakvec K); \osups\bigr)$   of the lattice  $\Halocs(\vakvec H,\vakvec K)= \bigl(\Halocs(\vakvec H,\vakvec K); \osubs\bigr)$  will be denoted by $\dual{\Halocs(\vakvec H,\vakvec K)}$.
By \czgsch{}\cite{rczgschjordhold}, or by Lemma~\ref{ldCoPSetwChAmsW}\eqref{ldCoPSetwChAmsc}, there is a unique diagram 
$\bigdiah{\dual{\Halocs(\vakvec H,\vakvec K)}}$
in $\bigDiag{\dual{\Halocs(\vakvec H,\vakvec K)}}$ whose left boundary chain and right boundary chain are  $\vakvec H$ and $\vakvec K$, respectively. 
Since $\pi=\pi_1$, the following remark is evident by \czgsch{}\cite{rczgschjordhold}.

\begin{remark}\label{whIsPefRjHOdr}
The unique permutation that establishes a down-and-up projective matching between the composition series $\vakvec H$ and $\vakvec K$ mentioned in the Introduction  is the permutation associated with 
$\bigdiah{\dual{\Halocs(\vakvec H,\vakvec K)}}$.
\end{remark}

\section{The main result}

Assume that $L$ is a slim semimodular lattice and $\diag L\in\Diag L$. 
Let $u\leq v$ be narrows.  If we reflect $\diag{[u,v]}$ vertically  while keeping the rest of the diagram $\diag L$ unchanged, we  obtain, as a rule, another planar diagram of $L$ that determines a  different permutation. In particular, if $u=0$ and $v=1$, then we obtain the permutation $\pi^{-1}$. Hence  we cannot associate a single well-defined partition with an \emph{abstract} slim semimodular lattice $L$, in general. That is why we need the following concept.

Let $\sigma\in S_n$, and let  $I=[u,v]=\set{u,\ldots, v}$ be an interval of the chain $\set{1<\cdots <n}$. If $\sigma(i)\in I$ holds for all $i\in I$, then we say that $I$ is \emph{closed} with respect to $\sigma$. The empty subset is also called closed.
If each of  $\set{1,\ldots, u-1}$, $I$ and $\set{v+1,\ldots,n}$ is closed with respect to $\sigma$ and $I\neq \emptyset$, then $I$  is called a \emph{section} of $\sigma$.  Sections that are minimal with respect to set inclusion are called \emph{segments} of $\sigma$. 
For brevity,  sections and segments of $\sigma$ are often called \emph{$\sigma$-sections} and \emph{$\sigma$-segments}.
Let $\Seg\sigma$ denote the \emph{set of all $\sigma$-segments}. We will prove soon that $\Seg\sigma$ is a partition on $\set{1,\ldots,n}$. For $i\in \set{1,\ldots,n}$, the unique segment that contains $i$ is denoted by $\pblokk i{\sigma}$. For example, if 
\begin{equation}\label{esPlepRmstS}
\sigma=
\begin{pmatrix} 1&2&3&4&5&6&7&8&9\cr
1&7&4&5&3&6&2&9&8
\end{pmatrix}=(27)(345)(89),
\end{equation}
then  $\Seg\sigma=\bigset{\set{1},
\set{2,3,4,5,6,7},\set{8,9}}$ and $\pblokk 8{\sigma}=\set{8,9}$.   
The restriction of $\sigma$ to a subset $I$ of $\set{1,\ldots,n}$ will be denoted by 
$\restrict \sigma I$.

Next, we define a binary relation on $S_n$.  Let $\sigma,\mu\in S_n$; we say that $\sigma$ and $\mu$ are \emph{sectionally inverted} or equal, in notation $(\sigma,\mu)\in\secie$, if $\Seg \sigma=\Seg\mu$ and, for all $I\in \Seg \sigma $, $\restrict \mu I\in\set{\restrict\sigma I, (\restrict\sigma I)^{-1}}$. (The letters $\pmb\rho$, $i$ and $e$ in the notation $\secie$ come from ``relation'', ``inverted'', and ``equal'', respectively.)
To shed more light on these concepts, we present an easy lemma.

\begin{lemma}\label{lMasecSeGm} Let $\sigma,\mu\in S_n$.
\begin{enumeratei}
\item\label{lMasecSeGa} $\Seg\sigma$ is a partition on $\set{1,\ldots,n}$.
\item\label{lMasecSeGb} The intersection of any two $\sigma$-sections is either a $\sigma$-section, or empty.
\item\label{lMasecSeGc} $\sigma$-sections are the same as $($non-empty$)$ intervals that are unions of $\sigma$-segments.
\item\label{lMasecSeGd}  $(\sigma,\mu)\in\secie$ if and only if there are pairwise disjoint $\sigma$-sections $J_1,\ldots,J_t$ such that $J_1\cup\cdots\cup J_t=\set{1,\ldots,n}$   
and, for $i=1,\ldots,t$, $\restrict \mu {J_i}\in\set{\restrict\sigma {J_i}, (\restrict\sigma {J_i})^{-1}}$.
\item\label{lMasecSeGe} $\secie$ is an equivalence relation on $S_n$.
\end{enumeratei}
\end{lemma}

For $\sigma\in S_n$, the  $\secie$-class of $\sigma$ will be denoted by $\blokk \sigma\secie$. 
 For example, if $\sigma$ is the permutation given in 
\eqref{esPlepRmstS}, then $\blokk \sigma\secie=\set{\sigma, \sigma^{-1}}$. Another example: if $n=7$,
then 
 $\blokk {(123)(567)}\secie=\bigset{(123)(567), (132)(567), (123)(576),(132)(576)}$. 
The quotient set $\bigset{\blokk \sigma\secie:\sigma\in S_n\!}$ will, of course, be denoted by $S_n/\secie$.

The \emph{class} of slim semimodular lattices of length $n$ will be denoted by $\smLat n$.
Let  $\isom$ denote isomorphism as a binary relation. For a lattice $L$, let 
$\izoclass L$ be the class of lattices isomorphic to $L$.
The quotient set 
\[\ismLat n:=\quotientset {\smLat n}{\isom}=\bigset{ \izoclass L: L\in \smLat n \!}\]
is called the \emph{set of isomorphism classes} of slim semimodular lattices of length $n$. Our goal is to establish a bijective correspondence between 
 $\ismLat n\,$ and $S_n/\secie$. This way, since we are interested in lattices only up to isomorphism,  slim semimodular lattices will be described by permutations. 

To accomplish our goal, we define four maps.
First of all, we need some notation. 
Consider the square grid $G$, see \eqref{egRidnostdnIon}. When there is no danger of confusion, we will simply write $\vee$ and $\wedge$ instead of $\joing$ and $\wedge_G$.
For $i,j\in\set{1,\ldots,n}$ and $u=c_i\vee d_j$, let $\jcongg(u)=\jcongg(c_i\vee d_j)$ denote the smallest join-congruence of $G$ that collapses $\set{c_{i-1}\vee d_j, c_i\vee d_{j-1}, c_i\vee d_j}$. 
Let 
\[\diag{\smLat n}:= \bigcup_{L\in \smLat n} \Diag L\text.
\] 
(It is a \emph{finite} set since $\ismLat n$ is finite.) 
Our maps are defined as follows.
\begin{definition}\label{mapsDefs}Let $n\in\mathbb N=\set{1,2,\ldots}$.
\begin{enumeratei}
\item\label{mapsDeffsa} For $\pi\in S_n$, let $\bbeta_\pi:=\bigvee_{i=1}^n \jcongg(c_i\vee d_{\pi(i)})$, in the congruence lattice of $(G;\vee)$. Then $\quotientalg G{\bbeta_\pi}$ is a lattice (not just a join-semilattice). By the \emph{canonical diagram} of $\quotientalg G{\bbeta_\pi}$ we mean its planar diagram $\diac{\quotientalg G{\bbeta_\pi}}$ such that 
\begin{align*}
\leftb{\diac{\quotientalg G{\bbeta_\pi}}}&=\set{\blokk{c_i}{\bbeta_\pi}: 0\leq i\leq n},\cr 
\rightb{\diac{\quotientalg G{\bbeta_\pi}}}&=\set{\blokk{d_i}{\bbeta_\pi}: 0\leq i\leq n}.
\end{align*}
(We will soon show that this makes sense.)  Let $\phi_0(\pi)=\diac{\quotientalg G{\bbeta_\pi}}$.  This defines a map 
$\phi_0\colon S_n\to \diag{\smLat n}$.
\item\label{mapsDeffsb} We define a map 
$\psi_0\colon \diag{\smLat n}\to S_n$ by $\diag L\mapsto \pi_{\diag L}$. 
\item\label{mapsDeffsc} Let $\phi\colon S_n/\secie\to \ismLat n$,  defined by  $\blokk \pi\secie\mapsto \izoclass{\biglat{\phi_0(\pi)}}$.
\item\label{mapsDeffsd}  Let $\psi\colon \ismLat n\to  S_n/\secie$, defined by  
$\izoclass L \mapsto \blokk{\psi_0({\diag L})}\secie = \blokk{\pi_{\diag L}}\secie$, where $\diag L$  denotes an arbitrarily chosen planar diagram of $L$.
\end{enumeratei}
Since $|S_0|=|S_0/\secie|=|\diag{\smLat 0}|=
|\ismLat 0| =1$,  the meaning of these maps for $n=0$ is obvious.
\end{definition}

\begin{theorem}\label{thmmAin} Slim semimodular lattices, up to isomorphism,  are characterized by 
permutations, up to the equivalence relation ``sectionally inverted or equal''. More exactly, $\phi_0$, $\phi$, $\psi_0$ and $\psi$ are well-defined maps, they are bijections, $\psi_0=\phi_0^{-1}$, and $\psi=\phi^{-1}$.
\end{theorem}

The case $n=0$ is trivial. In what follows, we always assume that $n\in \mathbb N$. The following result is an evident consequence of Theorem~\ref{thmmAin}, Remark~\ref{whIsPefRjHOdr}, and the fact that each lattice  is determined by any diagram of its dual lattice.

\begin{corollary}\label{corcOmpseDeiTmsnd}
$\pi$ from Theorem~\ref{jhbBNbeMbzTk}
 determines the lattice 
$\Halocs(\vakvec H,\vakvec K)$, that is  $\bigl({\Halocs(\vakvec H,\vakvec K); \osubs\bigr)}$, up to lattice isomorphism.
\end{corollary}

Theorem~\ref{thmmAin} will make the proof of the next corollary  quite easy.

\begin{corollary}\label{RcThTwocLsSuQal}
For each slim semimodular lattice $L$, there exist a finite cyclic group $G$ and composition series $\vakvec H$ and $\vakvec K$ of $G$ such that $L$ is isomorphic to the lattice
$\dual{\Halocs(\vakvec H,\vakvec K)}$. 
Conversely, for all groups $G$ with finite composition length and for any composition series $\vakvec H$ and $\vakvec K$ of $G$, 
$\dual{\Halocs(\vakvec H,\vakvec K)}$ 
is a slim semimodular lattice.
\end{corollary}

\begin{remark}\label{reMaRidke} Associated with a permutation $\pi \in S_n$, it is convenient to consider the \emph{grid matrix} $A(\pi ):=(G;\graf\pi )$ of $\pi $, where $\graf\pi :=\set{\celldn{c_i\joing d_{\pi (i)}}: 1\leq i\leq n}$. That is, $A(\pi )$ consists of the grid together with $n$ $4$-cells determined by $\pi $. In Figure~\ref{figureone}, the elements of $\graf\pi $ are  shaded grey. 
We can use grid matrices to clarify  the definition of $\phi_0$  as follows. 
For a $4$-cell $B$ of $G$, let $\jcongg(B)$ denote  $\jcongg(1_B)$. Equivalently, $\jcongg(B)$  is the smallest join-congruence of $G$ that collapses the upper edges of $B$. Then $\bbeta_\pi=\bigvee_{B\in \graf\pi} \jcongg(B)$ and $\phi_0(\pi)=\diac{\quotientalg G{\bbeta_\pi}}$. 
\end{remark}

\begin{remark}\label{reMaiWrGRidk} 
We can use $\graf \pi$ also in connection  with Definition~\ref{defpermmtxos}.  Indeed, for   $\diag L\in \diag{\smLat n}$, $\pi_{\diag L}=\pi_3$ is defined by the property $\graf{\pi_{\diag L}\kern-2pt}= \scells {\bbeta_{\diag L}}$.
\end{remark}

\section{Auxiliary lemmas and the proof of the main result}

\begin{proof}[Proof of Lemma~\ref{lMasecSeGm}]
For an interval $J=\set{u,\ldots,v}$ of $\set{1,\ldots,n}$, we define 
\begin{equation}\label{sllrNoeshM}
J_l=\set{1,\ldots,u-1}\,\text{ and }\,J_r=\set{v+1,\ldots,n}\text. 
\end{equation}
Assume that $I$ and $I'$ are sections of $\sigma$. Then the sets $I, I_l, I_r, I', I'_l$ and $I'_r$ are $\sigma$-closed. Let $J=I\cap I'$, and suppose that it is non-empty. Since $J_l\in\set{I_l,I'_l}$ and 
$J_r\in\set{I_r,I'_r}$, the sets $J, J_l$ and $J_r$ are $\sigma$-closed. Hence $J$ is a section of $\sigma$, proving part \eqref{lMasecSeGb}.

To prove  \eqref{lMasecSeGa}, let $a\in\set{1,\ldots,n}$. By part \eqref{lMasecSeGb}, there is a minimal $\sigma$-section $I=\set{u,\ldots,v}$ such that $a\in I$. Suppose that $I$ is not a $\sigma$-segment. Then there is a $\sigma$-segment $J=\set{u',\ldots,v'}$ such that $a\notin J\subset I$. We know that $u\leq a\leq v$, but  $a<u'$ or $v'<a$. We can assume that $a<u'$ since the case $v'<a$ can be treated similarly. Let $K=\set{u,\ldots,u'-1}$, and note that $a\in K$. Since intersections and unions of $\sigma$-closed subsets are $\sigma$-closed, we conclude that 
$K=I\cap J_l$, $K_l=I_l$ and $K_r=J\cup J_r$ are $\sigma$-closed. Hence $K$ is a $\sigma$-section, which contradicts $a\in K$ and the minimality of $I$. Consequently, each $a\in\set{1,\ldots,n}$ belongs to a $\sigma$-segment. Since distinct $\sigma$-segments are obviously disjoint by part \eqref{lMasecSeGb}, part \eqref{lMasecSeGa} follows.

Part \eqref{lMasecSeGc} is an evident consequence of parts  \eqref{lMasecSeGa} and \eqref{lMasecSeGb}.

The ``only if'' direction of part  \eqref{lMasecSeGd} is obvious since we can choose $\set{J_1,\ldots,J_t}:=\Seg \sigma$. To prove the ``if'' direction, assume that there are $\sigma$-sections $J_1,\ldots,J_t$ described in part \eqref{lMasecSeGd}. 
Let $I$ be a non-empty $\sigma$-closed subset of $\set{1,\ldots,n}$. For every $i\in \set{1,\ldots,t}$, $I\cap J_i$ is $\sigma$-closed. Hence it is $\mu$-closed since $\restrict \mu{J_i}\in\set {\restrict \sigma{J_i}, (\restrict \sigma{J_i})^{-1} }$. So $I=(I\cap J_1)\cup\cdots\cup (I\cap J_t)$ is $\mu$-closed. 
This implies that $\sigma$-sections are also $\mu$-sections. In particular, $J_1,\ldots,J_t$ are $\mu$-sections, which makes the role of $\sigma$ and $\mu$ symmetric. Therefore, $\mu$-sections are the same as $\sigma$-sections, and we conclude that $\Seg \sigma=\Seg\mu$.

Next, let $I\in\Seg\sigma=\Seg\mu$. Then there is an $i\in\set{1,\ldots,t}$ such that $I\cap J_i$ is non-empty. Since 
$I\cap J_i$ is a $\sigma$-section by part \eqref{lMasecSeGb} and 
$I$ is a minimal $\sigma$-section,  
$I\subseteq J_i$. Hence $\restrict \mu I=\restrict {(\restrict \mu {J_i})}I$ belongs to $\set{\restrict \sigma I,  (\restrict \sigma I)^{-1}  }$. Thus,  $(\sigma,\mu)\in\secie$, proving part~\eqref{lMasecSeGd}. 

Finally, part \eqref{lMasecSeGe} is  obvious. 
\end{proof}

\begin{lemma}[{\czgonly\cite[(14)+Cor.\ 22]{r:czg-mtx}}]\label{lMbETapIdesCribs} Assume that $i,j\in\set{1,\ldots,n}$, and let $\pi\in S_n$.
\begin{enumeratei}
\item\label{lMbETapIjJCribsa}  Then
$(c_{i-1}\vee d_j, c_i\vee d_j)\in \bbeta_\pi$
if{}f $\pi(i)\leq j$. Similarly, $(c_{i}\vee d_{j-1}, c_i\vee d_j)\in \bbeta_\pi$
if{}f $\pi^{-1}(j)\leq i$.
\item\label{lMbETapIjJCribsb} Equivalently, $(c_{i-1}\vee d_j, c_i\vee d_j)\in \bbeta_\pi$
if{}f  $\celldn{c_i\vee d_t}\in\graf\pi$ for some $t\in\set{1,\ldots,j}$.  Similarly, $(c_{i}\vee d_{j-1}, c_i\vee d_j)\in \bbeta_\pi$
if{}f  $\celldn{c_t\vee d_j}\in\graf\pi$ for some $t\in\set{1,\ldots,i}$. 
\item\label{lMbETapIjJCribsc} In particular, $(c_{i-1},c_i)\notin \bbeta_\pi$ and $(d_{j-1},d_j)\notin \bbeta_\pi$.
\end{enumeratei} 
\end{lemma}

\begin{lemma}\label{eLleGTbgRd} If $\pi\in S_n$, then $\quotientalg G{\bbeta_\pi}\in \smLat n$.
\end{lemma}

\begin{proof} Suppose that $\bbeta_\pi$ is not cover-preserving. Then, by \eqref{ecoprecoiffnoforbs}, there are $i,j\in\set{1,\ldots,n}$ such that $\celldn{c_i\vee d_j}$ is a $\bbeta_\pi$-forbidden $4$-cell of $G$. By symmetry, we can assume that $(c_i\vee d_{j-1}, c_i\vee d_j)\in\bbeta_\pi$. By Lemma~\ref{lMbETapIdesCribs}, $\pi^{-1}(j)\leq i$. 
Since $\celldn{c_i\vee d_j}$ is  $\bbeta_\pi$-forbidden, $(c_{i-1}\vee d_{j-1}, c_{i-1}\vee d_j)\notin\bbeta_\pi$. Using Lemma~\ref{lMbETapIdesCribs} again, we obtain that $\pi^{-1}(j)\not\leq i-1$. Hence $\pi^{-1}(j)= i$, that is $\pi(i)=j$. 
Again by  Lemma~\ref{lMbETapIdesCribs}, we infer that $(c_{i-1}\vee d_j, c_i\vee d_j)\in \bbeta_\pi$, which is a contradiction since $\celldn{c_i\vee d_j}$ is a $\bbeta_\pi$-forbidden $4$-cell. This proves that $\bbeta_\pi$ is a cover-preserving join-congruence. Since quotient lattices of finite semimodular lattices modulo cover-preserving join-congruences are semimodular by G.\ Gr\"atzer and E.\ Knapp~
\cite[Lemma 16]{r:GratzerKnapp1}, it follows that $\quotientalg G{\bbeta_\pi}$ is semimodular. Obviously (see also 
\cite[first paragraph of Section 2]{r:czg-mtx}), 
 slimness is preserved under forming join-homomorphic images, whence $\quotientalg G{\bbeta_\pi}$ is slim.  
The rest of the proof is also based on Lemma~\ref{lMbETapIdesCribs}. Since $\blokk{c_{i-1}}{\bbeta_\pi}\neq \blokk{c_i}{\bbeta_\pi}$, $\blokk{d_{i-1}}{\bbeta_\pi}\neq \blokk{d_i}{\bbeta_\pi}$, and $\bbeta_\pi$ is cover-preserving, 
\begin{equation}\label{eThEciperBtacoChns}
\blokk 0{\bbeta_\pi}=\blokk{c_0}{\bbeta_\pi}\prec \cdots\prec\blokk{c_n}{\bbeta_\pi}
,\qquad
\blokk 0{\bbeta_\pi}=\blokk{d_0}{\bbeta_\pi}\prec \cdots\prec\blokk{d_n}{\bbeta_\pi}
\text. 
\end{equation}
It follows from $\pi^{-1}(j)\leq n$ that
$\bblokk{c_n\vee d_{j-1}}{\bbeta_\pi}= \bblokk{c_n\vee d_{j}}{\bbeta_\pi}$ for all $j\in \set{1,\ldots,n}$. By transitivity, 
$\blokk{c_n}{\bbeta_\pi}= \bblokk{c_n\vee d_{0}}{\bbeta_\pi} = \bblokk{c_n\vee d_{n}}{\bbeta_\pi}= \blokk{1}{\bbeta_\pi}$. Hence $\blength {\quotientalg G{\bbeta_\pi}}=n$, and $\quotientalg G{\bbeta_\pi}\in \smLat n$.
\end{proof}

\begin{lemma}\label{lmKoooszeRe}
For every $k\in\set{1,\ldots,n}$, $(c_k,d_k)\in\bbeta_\pi$ if{}f $k$ is the largest element of  $\pblokk k{\pi}$.
\end{lemma}

\begin{proof} Assume that $k$ is the largest element of  $\pblokk k{\pi}$.
Using  the notation \eqref{sllrNoeshM}, it follows that  
$\set{1,\ldots,k}  =\pblokk k{\pi} \cup \bigl(\pblokk k{\pi}\bigr)_l  $ is  closed with respect to $\pi$ and  $\pi^{-1}$. Hence, by Lemma~\ref{lMbETapIdesCribs}, $(c_{i-1}\vee d_k, c_i\vee d_k)\in \bbeta_\pi$ and $(c_{k}\vee d_{j-1}, c_k\vee d_j)\in \bbeta_\pi$ for all $i,j\in\set{1,\ldots,k}$. Thus, we conclude that 
$(c_k,d_k)= (c_k\vee d_0, c_0\vee d_k)\in\bbeta_\pi$ by transitivity.

Conversely, assume that $(c_k,d_k)\in\bbeta_\pi$.
Denote $\set{1,\ldots,k}$ by $I$.
Since $\blokk {d_k}{\bbeta_\pi}$ is a convex join-subsemilattice containing $c_k$ and  $c_k\vee d_k$, it follows from $d_k\leq  c_{i-1}\vee d_k \leq  c_i\vee d_k \leq c_k\vee d_k$ that $(c_{i-1}\vee d_k, c_i\vee d_k)\in \bbeta_\pi$ for all $i\in I$. 
This implies that $\pi(i)\in I$, for all $i\in I$, by Lemma~\ref{lMbETapIdesCribs}. 
That is, $I$ is a $\pi$-closed subset of $\set{1,\ldots,n}$. Since then $I_r=\set{1,\ldots,n}\setminus I$ and $I_l=\emptyset$ are also $\pi$-closed,  $I$ is a $\pi$-section. Hence $k$, the largest element of $I$, is  the largest element of $\pblokk k{\pi}$ by \eqref{lMasecSeGa} and \eqref{lMasecSeGc} of  Lemma~\ref{lMasecSeGm}.
\end{proof}

\begin{lemma}\label{lIHcidismbcHin} 
$\bound{\quotientalg G{\bbeta_\pi}}= \bigset{\blokk{c_i}{\bbeta_\pi}: i\in\set{0,\ldots,n}}\cup \bigset{\blokk{d_i}{\bbeta_\pi}: i\in\set{0,\ldots,n}}$. 
\end{lemma}

\begin{proof} 
Let $K:= \bigset{\blokk{c_i}{\bbeta_\pi}: i\in\set{0,\ldots,n}}\cup \bigset{\blokk{d_i}{\bbeta_\pi}: i\in\set{0,\ldots,n}}$. 
The height of an element $y$, that is, the length of $[0,y]$, will be denoted by $h(y)$. 

To show that $K\subseteq \bound{\quotientalg G{\bbeta_\pi}}$, we prove by induction on $i$ that 
\begin{equation}
 \blokk{c_i}{\bbeta_\pi} \in  \bound{\quotientalg G{\bbeta_\pi}}
,\qquad
\blokk{d_i}{\bbeta_\pi} \in  \bound{\quotientalg G{\bbeta_\pi}}\text.\tag{\text{H$_i$}}
\end{equation} 
Condition (H$_0$) is obvious. Assume that $0<i\leq n$ and  (H$_{i-1}$) holds. By symmetry, it suffices to show that $\blokk{c_i}{\bbeta_\pi} \in  
\bound{\quotientalg G{\bbeta_\pi}}$. We can assume that $\blokk{c_i}{\bbeta_\pi} \notin \Jir{ ({\quotientalg G{\bbeta_\pi}})}$ since otherwise  \eqref{ejirsubbound} applies.
Hence, by   \eqref{eThEciperBtacoChns}, there
exists an element $c_s\vee d_t\in G$ such that 
 $\blokk{c_{i-1}}{\bbeta_\pi}\parallel \bblokk{c_{s}\vee d_t}{\bbeta_\pi} < \blokk{c_{i}}{\bbeta_\pi}$. Clearly, $s<i-1$. Hence 
$\blokk{c_{i}}{\bbeta_\pi}=\blokk{c_{i-1}}{\bbeta_\pi} \vee \bblokk{c_{s}\vee d_t}{\bbeta_\pi}= \blokk{c_{i-1}}{\bbeta_\pi} \vee \blokk{d_{t}}{\bbeta_\pi}$. Suppose that $t$ is minimal with respect to the property $\blokk{c_{i}}{\bbeta_\pi}=\blokk{c_{i-1}}{\bbeta_\pi} \vee \blokk{d_{t}}{\bbeta_\pi}$. Clearly, $t\geq 1$. Let $x:=c_{i-1}\vee d_{t-1}$. 
By the minimality of $t$, we obtain the inequalities
$\blokk{c_{i-1}}{\bbeta_\pi} \leq   \blokk{x}{\bbeta_\pi}   = 
\blokk{c_{i-1}}{\bbeta_\pi} \vee \blokk{d_{t-1}}{\bbeta_\pi} < \blokk{c_{i}}{\bbeta_\pi}$. 
Hence   \eqref{eThEciperBtacoChns} yields that 
$\blokk{x}{\bbeta_\pi} = \blokk{c_{i-1}}{\bbeta_\pi}$. 

Next, assume that $z\in G$ such that $\blokk{c_{i-1}}{\bbeta_\pi} < \blokk{z}{\bbeta_\pi}$. Then 
\[\blokk{z}{\bbeta_\pi} = \blokk{z}{\bbeta_\pi} \vee \blokk{c_{i-1}}{\bbeta_\pi} = \blokk{z}{\bbeta_\pi} \vee \blokk{x}{\bbeta_\pi} =  \bblokk{x\vee z}{\bbeta_\pi}\text.\] 
Since $\blokk{x}{\bbeta_\pi}= \blokk{c_{i-1}}{\bbeta_\pi} \neq \bblokk{x\vee z}{\bbeta_\pi}$, we obtain that $x<x\vee z$. Since $c_i\vee d_{t-1}$ and $c_{i-1}\vee d_{t}$ are the only covers of $x$, we conclude that $c_{i-1}\vee d_{t}\leq x\vee z$  or $c_i\vee d_{t-1}\leq x\vee z$. 
In the first case, $\blokk{c_{i}}{\bbeta_\pi} = \bblokk{c_{i-1}\vee d_{t}}{\bbeta_\pi}\leq \bblokk{x\vee z}{\bbeta_\pi}=\blokk{z}{\bbeta_\pi}$. 
In the second case, $\blokk{c_{i}}{\bbeta_\pi} =
\blokk{c_{i}}{\bbeta_\pi} \vee \blokk{c_{i-1}}{\bbeta_\pi} = \blokk{c_{i}}{\bbeta_\pi} \vee \blokk{x}{\bbeta_\pi}=\bblokk{c_{i}\vee x}{\bbeta_\pi}=\bblokk{c_{i}\vee d_{t-1}}{\bbeta_\pi} \leq \bblokk{x\vee z}{\bbeta_\pi} = \blokk{z}{\bbeta_\pi}$. This shows that  $\blokk{c_{i}}{\bbeta_\pi}$ is the only cover of $\blokk{c_{i-1}}{\bbeta_\pi}$. 
Take a diagram $\diag{(\quotientalg G{\bbeta_\pi})} \in \Diag{\quotientalg G{\bbeta_\pi}}$.
By left-right symmetry  and the induction hypothesis (H$_{i-1}$), we can assume that  $\blokk{c_{i-1}}{\bbeta_\pi}\in \leftb{\diag{(\quotientalg G{\bbeta_\pi})}}$. 
We know that $\leftb{\diag{(\quotientalg G{\bbeta_\pi})}}$ is a maximal chain. Hence $\blokk{c_{i}}{\bbeta_\pi}$, which is the only cover of $\blokk{c_{i-1}}{\bbeta_\pi}$, belongs to $\leftb{\diag{(\quotientalg G{\bbeta_\pi})}}$. Thus, $\blokk{c_{i}}{\bbeta_\pi}\in \bound{\diag{(\quotientalg G{\bbeta_\pi})}} =  \bound{\quotientalg G{\bbeta_\pi}}$, and (H$_i$) holds.
Therefore, $K\subseteq \bound{\quotientalg G{\bbeta_\pi}}$.

To show the converse inclusion, let us assume that 
\[\blokk{x}{\bbeta_\pi}\in    \bound{\quotientalg G{\bbeta_\pi}}=
\leftb{\diag{(\quotientalg G{\bbeta_\pi})}}\cup 
\rightb{\diag{(\quotientalg G{\bbeta_\pi})}}\text.
\]
Let  $\blokk{x}{\bbeta_\pi} \in \leftb{\diag{(\quotientalg G{\bbeta_\pi})}}$; the other case is similar. Denote $h(\blokk{x}{\bbeta_\pi})$ by $i$. 
Assume first that $\blokk{x}{\bbeta_\pi}$ belongs also to $\rightb{\diag{(\quotientalg G{\bbeta_\pi})}}$. 
 Then $\blokk{x}{\bbeta_\pi} \in \Nar {\quotientalg G{\bbeta_\pi}}$.
Hence $\blokk{x}{\bbeta_\pi}$ is comparable with $\blokk{c_i}{\bbeta_\pi}$. But $h(\blokk{c_i}{\bbeta_\pi})=i=h(\blokk{x}{\bbeta_\pi})$ by \eqref{eThEciperBtacoChns}, 
whence $\blokk{x}{\bbeta_\pi} = \blokk{c_i}{\bbeta_\pi}
\in K$.

Secondly, we assume  that $\blokk{x}{\bbeta_\pi} \notin \rightb{\diag{(\quotientalg G{\bbeta_\pi})}}$. 
Let $\blokk{y}{\bbeta_\pi}$ be the unique element of $\rightb{\diag{(\quotientalg G{\bbeta_\pi})}}$ with height $i$. Then $\blokk{x}{\bbeta_\pi}$ and 
$\blokk{y}{\bbeta_\pi}$ are the only elements of $\bound {\quotientalg G{\bbeta_\pi}}$ with height $i$, and they are distinct. Hence  $\Nar{\quotientalg G{\bbeta_\pi}}$ has  no element with height $i$.
Clearly, $\Jir{(\quotientalg G{\bbeta_\pi})}\subseteq K$. Hence $\blokk{c_i}{\bbeta_\pi} \neq \blokk{d_i}{\bbeta_\pi}$ since otherwise 
$\blokk{c_i}{\bbeta_\pi}$ would belong to $\Nar{\quotientalg G{\bbeta_\pi}}$ and it would be of height $i$ by 
\eqref{eThEciperBtacoChns}.
So $K$ also has two elements of height $i$, namely, $\blokk{c_i}{\bbeta_\pi}$ and $\blokk{d_i}{\bbeta_\pi}$. Since $K\subseteq \bound{\quotientalg G{\bbeta_\pi}}$, we conclude that $\set{\blokk{c_i}{\bbeta_\pi}, \blokk{d_i}{\bbeta_\pi}}=\set{\blokk{x}{\bbeta_\pi},\blokk{y}{\bbeta_\pi}}$. Hence 
$\blokk{x}{\bbeta_\pi}\in \set{\blokk{c_i}{\bbeta_\pi}, \blokk{d_i}{\bbeta_\pi}} \subseteq K$, proving that 
$\bound{\quotientalg G{\bbeta_\pi}}\subseteq K$.
\end{proof}

\begin{lemma}\label{lRtoSnlegsmgTzs} Assume that $I:=\set{u+1,\ldots,v}$ is a section of $\pi\in S_n$, and let $\sigma=\restrict\pi{I}$ be the restriction of $\pi$ to $I$. Then 
the subdiagram  $\diavarc{\bigl[ \blokk{c_u}{\bbeta_\pi}, \blokk{c_v}{\bbeta_\pi}\bigr]}$ of $\diac{\quotientalg G{\bbeta_\pi}}= \phi_0(\pi)$  equals $\phi_0(\sigma)$. Furthermore, 
$\blokk{c_u}{\bbeta_\pi}$ and $\blokk{c_v}{\bbeta_\pi}$ belong to $\Nar{\phi_0(\pi)}$.
\end{lemma}

\begin{proof} 
Consider the interval  $B:=[c_u\vee d_u,c_v\vee d_v]$ of $G$. Then $B$ is a square grid, a subgrid of $G$. We infer from Lemma~\ref{lMasecSeGm}\eqref{lMasecSeGa} and \eqref{lMasecSeGc} that $v$ is the largest element of 
$\pblokk v{\pi}$, and the same holds for $u$ if $u>0$. By  Lemma~\ref{lmKoooszeRe},  this 
yields that  $\blokk{c_u}{\bbeta_\pi}=\blokk{d_u}{\bbeta_\pi}= \bblokk{c_u\vee d_u}{\bbeta_\pi}$ and $\blokk{c_v}{\bbeta_\pi}=\blokk{d_v}{\bbeta_\pi}= \bblokk{c_v\vee d_v}{\bbeta_\pi}$. Hence Lemma~\ref{lIHcidismbcHin} implies the last sentence of Lemma~\ref{lRtoSnlegsmgTzs}.
Each element  $\blokk{y}{\bbeta_\pi}\in \bigl[ \blokk{c_u}{\bbeta_\pi}, \blokk{c_v}{\bbeta_\pi}\bigr]$ is of the form $ \blokk{x}{\bbeta_\pi}$ for some $x\in B$ since 
\[
\blokk{y}{\bbeta_\pi}=\bigl(\blokk{y}{\bbeta_\pi}
\vee \bblokk{ c_u\vee d_u}{\bbeta_\pi}\bigr)\wedge  \bblokk{ c_v\vee d_v}{\bbeta_\pi}= \bbblokk{ (y\vee c_u\vee d_u)\wedge (c_v\vee d_v)}{\bbeta_\pi}\text.
\]
Hence, as in  the Third Isomorphism Theorem in S.\ Burris and H.\,P.\ Sankappanavar \cite[Thm.\ 6.18]{r-burrissankap}, it is straightforward to see that 
 $\bigl[ \blokk{c_u}{\bbeta_\pi}, \blokk{c_v}{\bbeta_\pi}\bigr]$ is isomorphic to $\quotientalg B{(\restrict{\bbeta_\pi}{B})}$ 
and $\blokk{x}{\bbeta_\pi}\mapsto \blokk{x}{(\restrict{\bbeta_\pi}B)}$ is an isomorphism. 
This yields that $\diavarc{[ \blokk{c_u}{\bbeta_\pi}, \blokk{c_v}{\bbeta_\pi}]}= \diac{\quotientalg B{(\restrict{\bbeta_\pi}{B})}}$. Hence it suffices to show that $\restrict{\bbeta_\pi}{B} =\bbeta_\sigma$. In fact,  it suffices to show that $\restrict{\bbeta_\pi}{B}$ and $\bbeta_\sigma$ collapse exactly the same prime intervals of $B$.  But this is a straightforward consequence of Lemma~\ref{lMbETapIdesCribs}.     
\end{proof}

\begin{lemma}\label{ldCoPSetwChAmsW} 
Assume that  $t\in \mathbb N=\set{1,2,\ldots}$, $M$  is a slim semimodular lattice 
with $\Nar M=\set{0=z_0<z_1<\cdots <z_t=1}$,   $\diag M\in \Diag M$, and  
$U$ and $V$ are maximal chains in $M$ such that $U\cup V=\bound M$. Then the following four assertions hold.
\begin{enumeratei}
\item\label{ldCoPSetwChAmsa} $\Nar M=U\cap V$.
In particular, $\leftb{\diag M} \cap \rightb{\diag M} = \Nar M$.
\item\label{ldCoPSetwChAmsb} If $M$ is indecomposable, then $\set{U,V}=
\set{\leftb {\diag M},\rightb {\diag M}}$.
\item\label{ldCoPSetwChAmsc} $M$ has a planar diagram $\diah M$ such that $\leftb {\diah M}=U$ and $\rightb {\diah M}=V$.
\item\label{ldCoPSetwChAmsd} All planar diagrams of $M$ can be obtained from $\diag M$ in the following way. Take a subset $H$ of $\set{1,\ldots,t}$,  reflect the interval $\diag{[z_{i-1},z_i]}$ of $\diag M$ 
vertically for all $i\in H$, and keep the other $\diag{[z_{i-1},z_i]}$ unchanged. Furthermore, each subset $H$ of $\set{1,\ldots,t}$ yields a member of $\Diag M$.
\end{enumeratei}
\end{lemma}

\begin{proof} 
Suppose that $z\in U\cap V$.  
Then $z\in\Nar M$  since $z$ is comparable with all elements of $\Jir M$ by \eqref{ejirsubbound}. Conversely, since every element of $\Nar M$ belongs to all maximal chains, $\Nar M\subseteq U\cap V$. This proves \eqref{ldCoPSetwChAmsa}.
 
Assume that $M$ is indecomposable. For the elements of the  boundary of $\diag M$, we use the notation introduced in \eqref{eleriboudelesnot}. Since all chains of $M$ are of the same length, we can write
$U$ and $V$ in the form $\set{0=u_0\prec u_1\prec\cdots\prec u_n=1}$ and  $\set{0=v_0\prec v_1\prec\cdots\prec v_n=1}$, respectively. By 
\eqref{eeMntcotMsttw} and symmetry, we can assume that $u_1=c_1$. We prove by induction on $i$ that $u_i=c_i$ and $v_i=d_i$. The case $i\in\set{0,1,n}$ is clear. 
Assume that $1<i<n$,  $u_{i-1}=c_{i-1}$, $v_{i-1}=d_{i-1}$ but, say  $u_{i}\neq c_{i}$. By part \eqref{ldCoPSetwChAmsa}, there are exactly two elements in $\bound M$ whose height is $i$. Therefore, $u_{i} = d_{i}$, $v_{i} = c_{i}$ and $c_i\neq d_i$. Since distinct elements of the same height are incomparable,
\begin{align*}
c_{i-1}&< c_{i-1} \vee d_{i-1} =  
 (c_{i-1}\wedge u_{i-1}) \vee(v_{i-1}\wedge d_{i-1}) \cr
& \leq  (c_i\wedge u_i) \vee(v_i\wedge d_i)
 =   (c_i\wedge d_i) \vee(c_i\wedge d_i) = c_i\wedge d_i <c_i\text.
\end{align*}
This  contradicts that $c_{i-1}\prec c_i$, proving part  \eqref{ldCoPSetwChAmsb} of the statement.

Part \eqref{ldCoPSetwChAmsb} trivially implies part \eqref{ldCoPSetwChAmsc} in the particular case when $M$ is indecomposable or $|M|=2$, that is, when $t\leq 1$. Otherwise, for $i=1,\ldots, t$, let $M_i=[z_{i-1},z_i]$, $U_i:=M_i\cap U$ and  $V_i:=M_i\cap V$. Applying the particular case  to each $i\in \set{1,\ldots,t}$, we obtain part \eqref{ldCoPSetwChAmsc}.

Finally, consider a diagram $\diai M\in\Diag M$. Let $i\in\set{1,\ldots,t}$. Since $[z_{i-1},z_i]$ is clearly indecomposable, part \eqref{ldCoPSetwChAmsb} implies that $\diai{[z_{i-1},z_i]}$  equals $\diag{[z_{i-1},z_i]}$ or we obtain $\diai{[z_{i-1},z_i]}$ from  $\diag{[z_{i-1},z_i]}$ by a vertical reflection. This proves the first half of part \eqref{ldCoPSetwChAmsd}. The rest is evident.
\end{proof}

\begin{lemma}\label{lMcuWeHlQtTt}
Let $\bgamma$ be a cover-preserving join-congruence of the square grid $G$ such that, for all $i\in \set{1,\ldots,n}$, 
$(c_{i-1},c_i)\notin \bgamma$ and $(d_{i-1},d_i)\notin \bgamma$. Then $\bgamma=\bigvee_{B\in \scells\bgamma}\jcongg(B)$.
\end{lemma}

\begin{proof}Denote $\bigvee_{B\in \scells\bgamma}\jcongg(B)$ by $\bdelta$.  Since $\jcongg(B)\leq \bgamma$ for all $B\in \scells\bgamma$, we have $\bdelta\leq \bgamma$. Hence, it  suffices to show that if $\bgamma$ collapses a covering pair  of $G$, then so does $\bdelta$.  
Let $c_{i-1}\vee d_j\prec c_i\vee d_j$ be a covering pair collapsed by $\bgamma$. (The other case, $c_i\vee d_{j-1}\prec c_i\vee d_j$, is similar.) 
We know from \czgonly\cite[(14)+Cor.\ 22]{r:czg-mtx} that 
\begin{equation}\label{aEmIsglwcuwGb}
(c_{i-1}\vee d_j, c_i\vee d_j)\in \bdelta
\,\text{ if{}f }\, \celldn{c_i\vee d_t}\in \scells\bgamma\text{ for some }t\in\set{1,\ldots,j}\text.
\end{equation}
Take the minimal $t$ such that $(c_{i-1}\vee d_t, c_i\vee d_t)\in \bgamma$. Clearly, $t\in\set{1,\ldots,j}$. Since $\celldn{c_i\vee d_t}$ cannot be a $\bgamma$-forbidden $4$-cell, it belongs to $\scells\bgamma$. By \eqref{aEmIsglwcuwGb}, this implies $(c_{i-1}\vee d_j, c_i\vee d_j)\in \bdelta$. 
\end{proof}

The next lemma uses the notation of Remark~\ref{reMaRidke}.

\begin{lemma}\label{LpPgrScllsbepi} Let $\pi\in S_n$. Then $\scells{\bbeta_\pi}=\graf\pi$. 
\end{lemma}

\begin{proof}.
For $i\in\set{1,\ldots,n}$, we say that  $\bigset{\celldn{c_i\vee d_t}: t\in\set{1,\ldots,n}}$ is a \emph{row of $4$-cells}. Obviously, for every join-congruence $\balpha $ of $G$, every row contains at most one source cell of $\balpha $. Hence $|\graf\pi|= n\geq |\scells{\bbeta_\pi}|$. 
On the other hand, it is straightforward to infer from Lemma~\ref{lMbETapIdesCribs}\eqref{lMbETapIjJCribsb} that $\graf\pi\subseteq  \scells{\bbeta_\pi}$.
\end{proof}

\begin{lemma}\label{lMpsinullaNdphnuOK} $\phi_0$ and $\psi_0$ are well-defined, and they are reciprocal bijections.
\end{lemma}

\begin{proof} We know from Lemma~\ref{eLleGTbgRd} that $\quotientalg G{\bbeta_\pi}\in \smLat n$. 
Therefore it follows from Lemma~\ref{ldCoPSetwChAmsW}\eqref{ldCoPSetwChAmsc} and 
Lemma~\ref{lIHcidismbcHin} that $\diac{\quotientalg G{\bbeta_\pi}}$ is well-defined.  Consequently,  $\phi_0$ is a well-defined $S_n\to \diag{\smLat n}$ map.
By Section~\ref{sectionthreeways} (for example, by Proposition~\ref{prdiagLpermalleqal} combined with  Lemma~\ref{lemdeftwoworks} or Lemma~\ref{defperMbymtx}), $\psi_0$ is a well-defined $\diag{\smLat n}\to S_n$ map.

Next, let $\diag L\in\diag{\smLat n}$, and denote $\biglat{\diag L}$ by $L$. 
The join-homomorphism $\eta$ from Definition~\ref{defpermmtxos} is surjective by \eqref{ejirsubbound}. Hence $L\cong \quotientalg  G{\bbeta_{\diag L}}$, and the map $\tilde\eta\colon \quotientalg G{\bbeta_{\diag L}} \to L$, defined by $\blokk x{\bbeta_{\diag L}}\mapsto \eta(x)$, is an isomorphism. 
Furthermore, $\pi:=\pi_3=\psi_0(\diag L)=\pi_{\diag L}$ is determined by 
$ \scells {\bbeta_{\diag L}}$. 
Combining the definition of $\phi_0$ and Remark~\ref{reMaiWrGRidk},  
\[\bbeta_\pi=\bigvee_{i=1}^n \jcongg(c_i\vee d_{\pi(i)}) = \bigvee\set{\jcongg(B)  :  B\in\graf\pi} =  \bigvee\set{\jcongg(B)  :  B\in  \scells {\bbeta_{\diag L}} }\text.\]
Hence, by \eqref{ecidjseMoDbEtakd} and Lemma~\ref{lMcuWeHlQtTt},  $\bbeta_\pi= 
\bbeta_{\diag L}$. Thus, $L\cong \quotientalg G{\bbeta_\pi}$. 
Since $\tilde\eta(\blokk{c_i}{\bbeta_\pi})=\eta(c_i)=c_i$, we obtain that $\tilde\eta\bigl(\leftb{ \diac{\quotientalg G{\bbeta_\pi}}}\bigr) = \leftb{\diag L}$. 
Similarly, we conclude that $\tilde\eta\bigl(\rightb{ \diac{\quotientalg G{\bbeta_\pi}}}\bigr) = \rightb{\diag L}$,  whence it follows that $\phi_0(\psi_0(\diag L))=\phi_0(\pi)=\diag L$. 

Next, let $\pi\in S_n$. Let $G$ be the corresponding square grid of length $2n$. We use the notation introduced in \eqref{egRidnostdnIon}. Let $\quotientalg G{\bbeta_\pi}$, $\diac{\quotientalg G{\bbeta_\pi}}$,   $\blokk{c_i}{\bbeta_\pi}$, and  $\blokk{d_j}{\bbeta_\pi}$  be denoted by $L$,  $\diag L$, 
$c_i'$, and $d_j'$, respectively.   
Observe that  $\leftb{\diag L}=\set{c_i':  0\leq i\leq n }$ and  $\rightb{\diag L}=\set{d_j':  0\leq j\leq n }$. Let $G'$  be their direct product. We identify the lower left  boundary and the lower right boundary of $G'$ with $\leftb{\diag L}$ and $\rightb{\diag L}$, respectively. Hence $G'=\bigset{c_i'\joinvg d_j':i,j\in\set{ 0,\ldots,n}}$. 

We have to consider three join-homomorphisms. Let $\gamma\colon G\to G'$, defined by $c_i\joing d_j\mapsto c_i'\joinvg d_j'$, be the first one.  The second one is  $\eta'\colon G'\to L$,  defined by 
$c_i'\joinvg d_j'\mapsto
c_i'\joinl d_j'=\blokk{c_i}{\bbeta_\pi} \joinl  \blokk{d_j}{\bbeta_\pi}= \bblokk{c_i\joing  d_j}{\bbeta_\pi}$. Let $\eta:= \eta'\circ\gamma$ be the third one, that is,  
\begin{equation}\label{EtPdcRprime}
\eta\colon G\to L,\quad 
\eta(c_i\joing d_j) = \eta'(\gamma(c_i\joing d_j)\bigr) = \bblokk{c_i\joing  d_j}{\bbeta_\pi}\text.
\end{equation}
According to Definition~\ref{defpermmtxos}, $\pi_{\diag L}$ is defined by the kernel of $\eta'$. 
Since the notation of grid elements in Definition~\ref{defpermmtxos} is irrelevant and $\gamma$ is an isomorphism, Definition~\ref{defpermmtxos} applied to $\eta$ yields the same permutation. 
By \eqref{EtPdcRprime}, $\Ker\eta$, the kernel of $\eta$, is $\bbeta_\pi$. 
Hence, we obtain from Remark~\ref{reMaiWrGRidk} that  $\psi_0(\diag L)=\pi_{\diag L}$ is the unique permutation that satisfies the equation  
$\graf{\bigl(\psi_0(\diag L)\bigr)} = 
\scells{\Ker \eta} =
\scells{\bbeta_\pi}$. By Lemma~\ref{LpPgrScllsbepi}, this is equivalent with $\graf{\bigl(\psi_0(\diag L)\bigr)}=\graf\pi$. Since $\pi$ instead of $\psi_0(\diag L)$ also satisfies this equation, we obtain that $\psi_0(\diag L)= \pi$.
Thus, $\pi=\psi_0(\diag L)=\psi_0(\phi_0(\pi))$. Therefore, $\phi_0$ and $\psi_0$ are reciprocal bijections.
\end{proof}

\begin{proof}[Proof of Theorem~\ref{thmmAin}]
Clearly, if $L_1,L_2\in \smLat n$ and $L_1\cong L_2$, then $\Diag {L_1}=\Diag {L_2}$. Hence, to show that $\psi$ is well-defined, it suffices to consider two diagrams of the \emph{same} lattice.
Assume that $L\in\smLat n$,  $\diag L,\diah L\in\Diag L$, and $|L|\geq 2$. Let $\Nar L=\set{0=z_0<z_1<\cdots <z_t=1}$, $t\in \mathbb N$. The height of $z_i$ will be denoted by $h_i$. 
It follows trivially from Definition~\ref{defpermmirr} or \ref{defpermone} that $I_i:=\set{h_{i-1}+1,\ldots,h_i}$ is both a $\pi_{\diag L}$-section and a $\pi_{\diah L}$-section. By Lemma~\ref{lRtoSnlegsmgTzs}, 
\begin{equation}\label{diWmMOsuVsK}
\restrict {\pi_{\diag L}} {I_i} = \pi_{\diag {[z_{i-1},z_i]}}
,\qquad
\restrict {\pi_{\diah L}} {I_i} = \pi_{\diah {[z_{i-1},z_i]}}\text.
\end{equation}
It follows  from, say, Definition~\ref{defpermone} that if we interchange the left and the right boundaries then we obtain the inverse permutation.  For each $i\in\set{1,\ldots,t}$,  Lemma~\ref{ldCoPSetwChAmsW}\eqref{ldCoPSetwChAmsd} permits only two cases: 
$\diah{[z_{i-1},z_i]}$ is obtained from $\diag {[z_{i-1},z_i]}$ by a vertical reflection or 
 $\diah{[z_{i-1},z_i]}=\diag {[z_{i-1},z_i]}$. In the first case, \eqref{diWmMOsuVsK} implies
$\restrict {\pi_{\diah L}} {I_i}= (\restrict {\pi_{\diag L}} {I_i})^{-1}$. In the second case,  
\eqref{diWmMOsuVsK} yields that $\restrict {\pi_{\diah L}} {I_i}= \restrict {\pi_{\diag L}} {I_i}$. Therefore, $\restrict {\pi_{\diah L}} {I_i}\in\set{\restrict {\pi_{\diag L}} {I_i}, (\restrict {\pi_{\diag L}} {I_i})^{-1} }$  for $i=1,\ldots, t$. Consequently, we derive from
Lemma~\ref{lMasecSeGm}\eqref{lMasecSeGd} that 
$\bigl(\psi_0(\diag L), \psi_0(\diah L) \bigr)=(\pi_{\diag L},\pi_{\diah L})\in\secie$. Thus, $\psi$ is a well-defined map.

Next, assume that $\pi,\sigma\in S_n$ such that $(\pi,\sigma)\in\secie$. We know that $\pi$ and $\sigma$ have the same segments. Let $0=j_0<\cdots<j_t=n$ such that $\Seg\pi=\Seg\sigma=\bigset{ \set{j_{r-1}+1,\ldots,j_r}: 1\leq r\leq t}$. Let $\mu\in\set{\pi,\sigma}$. Then 
\begin{equation}\label{eZidkdscuhrmSt}
\Nar {\phi_0(\mu)} = \bigset{\blokk {c_{j_r}}{\bbeta_\mu}: 0\leq r\leq t}
\end{equation}
 by \eqref{eThEciperBtacoChns} and Lemmas~\ref{lmKoooszeRe} and \ref{lIHcidismbcHin}. 
Consider an $r\in\set{1,\ldots,t}$. For brevity,  let $I=\set{j_{r-1}+1,\ldots,j_r}$, the $r$-th segment of $\pi$ and $\sigma$.
We know that $\restrict\sigma I\in \set{\restrict\pi I, (\restrict\pi I)^{-1}}$.
We obtain from Lemma~\ref{lRtoSnlegsmgTzs} that 
$\diac{\bigl[ \blokk {c_{j_{r-1}}}{\bbeta_\mu}, \blokk {c_{j_r}}{\bbeta_\mu}\bigr]}=\phi_0({\restrict\mu I})$. 
Hence if $\restrict\sigma I= \restrict\pi I$, then 
$\diac{\bigl[ \blokk {c_{j_{r-1}}}{\bbeta_\pi}, \blokk {c_{j_r}}{\bbeta_\pi}\bigr]}= \diac{\bigl[ \blokk {c_{j_{r-1}}}{\bbeta_\sigma}, \blokk {c_{j_r}}{\bbeta_\sigma}\bigr]}$, implying that 
\begin{equation}\label{isUeEiwcHkK}
\bigl[ \blokk {c_{j_{r-1}}}{\bbeta_\pi}, \blokk {c_{j_r}}{\bbeta_\pi}\bigr] \cong \bigl[ \blokk {c_{j_{r-1}}}{\bbeta_\sigma}, \blokk {c_{j_r}}{\bbeta_\sigma}\bigr]\text.
\end{equation}
Otherwise, assume that $\restrict\sigma I= (\restrict\pi I)^{-1}$. Therefore, when Definition~\ref{mapsDefs}\eqref{mapsDeffsa} is applied to  $\restrict\sigma I$ and $\restrict\pi I$,  the role of the $c_i$ and that of the $d_i$ are interchanged. Consequently, $\diac{\bigl[ \blokk {c_{j_{r-1}}}{\bbeta_\sigma}, 
\blokk {c_{j_r}}{\bbeta_\sigma}\bigr]}$ is obtained from 
$\diac{\bigl[ \blokk {c_{j_{r-1}}}{\bbeta_\pi}, \blokk {c_{j_r}}{\bbeta_\pi}\bigr]}$ by a vertical reflection, and \eqref{isUeEiwcHkK} holds again. 
From \eqref{isUeEiwcHkK}, applied for $r=1,\ldots,t$, and \eqref{eZidkdscuhrmSt}, we obtain that $\biglat{\phi_0(\pi)}\cong \biglat{\phi_0(\sigma)}$. Thus,  $\phi$ is a well-defined map.

Finally, since $\phi_0$ and $\psi_0$ are reciprocal bijections by Lemma~\ref{lMpsinullaNdphnuOK}, so are the maps $\phi$ and $\psi$.
\end{proof}

\begin{proof}[Proof of Corollary~\ref{RcThTwocLsSuQal}]
As detailed in the Introduction, the second part of the statement is known. 
By Theorem~\ref{thmmAin}, it suffices to show that for each $\pi\in S_n$ there exist a finite cyclic group $G$ and composition series $\vakvec H$ and $\vakvec K$ of $G$ such that the unique permutation $\sigma$ associated with 
$\diah{(\dual{\Halocs(\vakvec H,\vakvec K)})}$, see  Remark~\ref{whIsPefRjHOdr},  equals $\pi$. Let $p_1,\ldots,p_n$ be distinct primes, and let $G$ be the cyclic group of order $p_1p_2\ldots p_n$. For $i=1,\ldots,n$, let $H_i$ and $K_i$ be the unique subgroup of order $p_1\dots p_i$ and $p_{\pi^{-1}(1)}\ldots p_{\pi^{-1}(i)}$, respectively. Then $|H_i/H_{i-1}|=p_i$ and 
$|K_j/K_{j-1}|=p_{\pi^{-1}(j)}$, for all $i,j\in\set{1,\ldots,n}$. 
 Since down-and-up projective quotients are isomorphic, $p_i=|H_i/H_{i-1}|$ equals $|K_{\sigma(i)}/K_{\sigma(i)-1}|$, which is $p_{\pi^{-1}(\sigma(i))}$. Hence $i=\pi^{-1}(\sigma(i))$, for all $i\in\set{1,\ldots,n}$, and we conclude that $\sigma=\pi$.
\end{proof}

\subsection{Note added on January 9, 2013}
Using \init{G.\ }Cz\'edli and \init{E.\,T.~}Schmidt~\cite[Theorem 1.3 and Lemma 2.7]{rczgschjordhold}, it is straightforward to prove that Definition~\ref{defpermone} yields the same permutation as the one defined for a particular case in \init{R.\,P.\ }Stanley~\cite{stanley}. It is routine to see that Definition~\ref{defpermmtxos} gives the same permutation as the one defined by  
\init{H.~}Abels~\cite{abelsgaldist}.
Theorem~\ref{thmmAin} strengthens \init{H.~}Abels~\cite[Remark 2.14]{abelsgaldist}, which asserts that a slim semimodular lattice is determined by the permutation associated with it. Our approach and terminology are different from those in \cite{abelsgaldist} and \cite{stanley}.

\end{document}